\documentclass[a4paper,leqno,12pt]{article} 
\usepackage{amsmath, amssymb}
\usepackage{amsthm} 
\usepackage{bm}
\theoremstyle{plain} 
\newtheorem{theorem}{Theorem}[section] 
\newtheorem{lemma}[theorem]{Lemma}
\newtheorem{cor}[theorem]{Corollary}

\theoremstyle{definition} 
\newtheorem{defi}[theorem]{Definition}
\newtheorem{remark}[theorem]{Remark}
\newtheorem{ex}[theorem]{Example}

\title{{\bf Infinite particle
systems 
of long range jumps 
with long range interactions} \footnote{
2010 {\textit{Mathematics Subject Classification}}. 
Primary 60K35; Secondary 60J75.}
\footnote{
{\textit{Key words and phrases}}. 
Interacting L\'{e}vy processes, infinitely particle systems, Dirichlet form, jump type,
logarithmic potential.}
\footnote{Partly supported by JSPS Grant-in-Aid for Scientific Research (A) No. 24244010. }} 

\author{Syota Esaki \thanks{Department of Mathematics ,Faculty of Science, Tokyo Institute of Technology, 152-8551,  Japan. email: esaki.s.ab@m.titech.ac.jp}}
\date{} 
\renewcommand\tilde{\widetilde}
\renewcommand\hat{\widehat}
\renewcommand\#{\sharp}

\def\R{\mathbb{R}}
\def\N{\mathbb{N}}

\def\e{\varepsilon}
\def\bda{\bm{a}}
\def\bdb{\bm{b}}

\def\P{{\mathbb P}}
\def\E{{\mathbb E}}

\def\frakM{{\mathfrak{M}}}
\def\sfX{{\Xi}}

\def\sfAn{\frakM_{\bda_n-\1}^{2(\bda_n)_++\1}}

\def\sfn{{\sf n}}
\def\sfm{{\sf m}}

\def\x{{\text{\boldmath $x$}}}

\def\1{{\bf 1}}

\newcommand{\chemu}{\check{\mu}}
\newcommand{\chesig}{\check{\sigma}}
\def\D{{\mathbb D}}
\def\Qr{U_{2^r}}

\newcommand{\chia}{\chi[\bda]}
\newcommand{\chian}{\chi[\bda_{n}]}

\def\Dc{\mathfrak{D}_{\circ}}
\def\Dinf{\mathfrak{D}_{\infty}}

\def\Eq{\mathfrak{E}^{\sfn, \sfm}_{r,\ell, q, k, \xi}}

\def\Eqban#1{\mathfrak{E}^{\sfn, \sfm, #1}_{r,\ell, q, k, \xi}}

\def\Eqm{\mathfrak{E}^{\sfm}_{\ell, q, k, \xi}}
\def\hiku#1{\langle #1 \rangle}

\def\munm{\check{\mu}_{r,\ell,k,\xi}^{\sf{n},\sf{m}}}

\newcommand{\bd}[1]{\mbox{\boldmath$#1$}}
\renewcommand\tilde{\widetilde}
\renewcommand\hat{\widehat}
\renewcommand\#{\sharp}

\newcommand\nadeffsa[2]{\nabla_{#1}^{#2}}

\makeatletter
 
 \@addtoreset{equation}{section}
\makeatother

\begin{document}

\maketitle

%%%%%%%%%%%%%%%%%%%%%%%%%%%%%%%%%%%%%%%%%%%%%%%%%%%%%%%%%%
%%%  ABST  %%%%%%%%%%%%%%%%%%%%%%%%%%%%%%%%%%%%%%%%%%%%%%
%%%%%%%%%%%%%%%%%%%%%%%%%%%%%%%%%%%%%%%%%%%%%%%%%%%%%%%%%%

\begin{abstract}
In this paper
a general theorem for constructing infinite particle systems of jump type with long range interactions is presented.  
It can be applied to the system that each particle undergoes an $\alpha$-stable process and interaction between particles is given by the logarithmic potential appearing random matrix theory or potentials of Ruelle's class with polynomial decay. It is shown that the system can be constructed for any $\alpha \in (0, 2)$ if its equilibrium measure $\mu$ is translation invariant, and $\alpha$ is restricted by the growth order of the 1-correlation function of the measure $\mu$ in general case.  
\end{abstract}

%%%%%%%%%%%%%%%%%%%%%%%%%%%%%%%%%%%%%%%%%%%%%%%%%%%%%%%%%%
%%%  SEC1   %%%%%%%%%%%%%%%%%%%%%%%%%%%%%%%%%%%%%%%%%%%%%%
%%%%%%%%%%%%%%%%%%%%%%%%%%%%%%%%%%%%%%%%%%%%%%%%%%%%%%%%%%
\section{Introduction} %%%%%%%%%%%%%%%%%%%%%%%%%%%%%%%%%%%
%%%%%%%%%%%%%%%%%%%%%%%%%%%%%%%%%%%%%%%%%%%%%%%%%%%%%%%%%%

The studies of infinite particle systems 
with interaction were started from around 1970's by Spitzer \cite{Spi69} and Liggett \cite{Lig77,Lig85}. 
They constructed systems of particles moving on lattices (e.g. the square lattice $\mathbb{Z}^d$)
by means of Feller processes on the configuration spaces, 
which are compact with the product topology.
In this paper we discuss infinite particle systems of jump type with interaction
on continuum spaces. In the case where the continuum space is the $d$-dimensional Euclidean space $\mathbb{R}^d$,
the configuration space is represented as
$$
\frakM=\{ \xi = \sum_i \delta_{x_i}; \xi(K)<\infty\mbox{ for all compact sets $K \subset \R^d$} \},
$$
where $\delta_a$ stands for the delta measure at $a$. We endow $\frakM$ with the vague topology. Then $\frakM$ is a Polish space and $\mathfrak{N}\subset \mathfrak{M}$ is relative compact if and only if $\sup_{\xi\in\mathfrak{N}}\xi(K)<\infty$ for any compact set $K\subset\mathbb{R}^d$. 
For $x,y\in \mathbb{R}^d$ and $\xi \in \mathfrak{M}$, we write $\xi^{xy}=\xi-\delta_x+\delta_y$ and $\xi\setminus x = \xi-\delta_x$ if $\xi(\{x\})\ge 1$.

An infinite particle system of jump type is characterized by its rate function $c(\xi, x; y)$, $(\xi,x,y)\in\mathfrak{M}\times\R^d\times\R^d$, which controls the jump rate from $x$ to $y$ under the configuration $\xi$.
We consider, in this paper, the case that the rate function is given by 
$c(\xi, x; y) = 0$ if $\xi(\{x\})=0$,
and
\begin{equation*}
c(\xi, x; y) = \nu(\xi, x; y) + \nu(\xi^{xy}, y; x) \frac{d\mu_y}{d\mu_x} (\xi \setminus x)\frac{\rho^1(y)}{\rho^1(x)}, \quad \mbox{ if }\xi(\{x\})\ge 1, %\label{defc}
\end{equation*}
with some positive measurable function $\nu$ on $\frakM \times \R^d \times \R^d$ and some probability measure $\mu$ on $\frakM$.
Here, $\mu_x$ is the reduced Palm measure defined by $\mu_x = \mu\left( \left. \cdot - \delta_{x} \right| \xi(\{ x \}) \geq 1 \right)$ for $x \in \R^d$, 
$\rho^1(x)$ is the $1$-correlation function of $\mu$  and $d\mu_y/d\mu_x$ is the Radon-Nikodym derivative of $\mu_y$ with respect to $\mu_x$. 
We then introduce the linear operator $L_0$ on the space of local smooth functions $\mathfrak{D}_0$ in (\ref{D0}) defined by 
\begin{equation*}
L_0 f(\xi) = \int_{\R^d} \xi(dx) \int_{\R^d} dy \; c(\xi, x; y)[f(\xi^{x y})-f(\xi)], %\label{Liggen}
\end{equation*}
and the associated bilinear form $\mathfrak{E}$ on $\mathfrak{D}_\infty$ in (\ref{Dinfty}) given by
\begin{equation}
\mathfrak{E}(f, g) = \frac{1}{2} \int_{\frakM} d\mu \int_{\R^d} \xi(dx) \int_{\R^d} \nu(\xi, x; y)\{ f(\xi^{x y})-f(\xi) \} \{ g(\xi^{x y})-g(\xi) \} dy. \label{LigDiri}
\end{equation}
For $R > 0$, the subset $\frakM_R=\{ \xi \in \frakM$ ; $| x_i-x_j | \geq R$ for $i \neq j$  $\}$ of $\frakM$, which is regarded as the configuration of hard balls, is compact with the vague topology. Then by using a slight modification of Liggett's theorem \cite{Lig85}, we can construct the Feller process generated by the closure of $L_0$ describing an interacting particle system of hard balls moving by random jump on $\R^d$  under suitable assumptions on the rate function $c$ \cite{Tan89}.
In this situation the rate function $c$ satisfies the following detailed balance condition: 
\begin{equation*}
c(\xi^{x y}, y; x) = c(\xi, x; y) \frac{\rho^1(x)}{\rho^1(y)} \frac{d\mu_x}{d\mu_y}(\xi \setminus y), \quad x, y \in \R^d. 
\end{equation*}
Hence, we see that $\mu$ is a reversible measure of the process and the closure of the bilinear form $(\mathfrak{E}, \Dinf)$ is the Dirichlet form associated with it. 
However, the above argument by Liggett's theorem can not be applied to construct the process on $\frakM$, since $\frakM = \frakM_0$ is not locally compact. 
%To conquer the difficulties we use the Dirichlet form theory. 

The diffusion processes on general Polish spaces, which may be non-locally compact, are constructed by the Dirichlet form theory (e.g. Kusuoka \cite{Kus82}, Ma-R\"{o}ckner \cite{MR},  Osada \cite{O96} and others). The infinite particle system of jump type with interaction was also constructed by Kondratiev-Lytvynov-R\"{o}ckner \cite{KLR}, Lytvynov-Ohlerich \cite{LO08}.
They treated the case that the reversible measure $\mu$ is a Gibbs measure in \cite{KLR} and a determinantal random point field in \cite{LO08}.
A determinatal random point field is associated with its correlation operator $K$ with 
$\mathrm{Spec}(K)\subset [0,1]$. Their result in \cite{LO08} excludes the case that $\mathrm{Spec}(K)$ contains $1$, 
which includes Sine, Airy, Bessel and Ginibre random point fields.

Let  $\Phi : \mathbb{R}^d \to \mathbb{R}\cup \{\infty\}$ be a self-potential and  $\Psi : \mathbb{R}^d\times\mathbb{R}^d\to \mathbb{R}\cup \{\infty\}$ an interaction potential with $\Psi(x,y)=\Psi(y,x)$.
Osada \cite{Osa13a,Osa13b} introduced a class of probability measures on $\frakM$ associated with $\Phi$ and $\Psi$, and called its element a {\it quasi-Gibbs measure} (Definition \ref{def 5.1}).
The class includes Gibbs measures of Ruelle's class and Sine, Airy, Bessel and Ginibre random point fields \cite{Osa13a,Osa13b}.
He constructed a diffusion process describing a system of infinite Brownian particles with the potentials $\Phi$ and $\Psi$ by Dirichlet form technique related to quasi-Gibbs measures, and showed in \cite{Osa12} that the diffusion process solves the infinite dimensional stochastic differential equation (ISDE) :
\begin{equation*} %\label{SDEBM}
dX_j(t)=dB_j(t)-\frac{1}{2}\nabla\Phi(X_j(t))dt-\frac{1}{2}\sum_{\substack{k=1\\ k\not=j}}^{\infty}\nabla\Psi(X_j(t),X_k(t))dt, j\in\mathbb{N}.
\end{equation*}

It is an interesting and natural problem to extend Osada's results to infinite particle systems in which each particle undergoes a L\'{e}vy process of jump type (e.g. Cauchy process). 
In this paper we study the construction of the processes describing infinite particle systems of jumps with the potentials $\Phi$ and $\Psi$, by Dirichlet form technique.
The related infinite dimensional stochastic differential equations are treated in the forthcoming paper \cite{EsaSDE}.

We make some assumptions in Section 2. 
Assumption (A.1) is the closability of the bilinear form $(\mathfrak{E}, \Dinf)$. 
For a long range interaction potential $\Psi$, such as the $\log$-potential, $c(\xi, x;y)$ is not generally well-defined for some $(\xi, x; y)$ because of the divergence of $d\mu_y/d\mu_x$. The closability of the bilinear form $(\mathfrak{E}, \Dinf)$ ensure that $c$ is well-defined $\mu$-almost surely.
Theorem \ref{CQG} states that $(\mathfrak{E}, \Dinf)$ is closable 
if $\mu$ is a quasi-Gibbs measure with assumption (A.4).
We denote the closure of $(\mathfrak{E}, \Dinf)$ by $(\mathfrak{E}, \mathfrak{D})$.
 Theorem \ref{Theorem.1}  states that under assumptions (A.1)--(A.2) and (B.0)--(B.4), $(\mathfrak{E}, \mathfrak{D})$ is a quasi-regular Dirichlet form. Therefore there exists a special standard process $(\sfX_t, \mathbb{P}_{\xi})$ associated with $(\mathfrak{E}, \mathfrak{D})$. 
Reminding that $\mathfrak{N}\subset \mathfrak{M}$ is relative compact if and only if $\sup_{\xi\in\mathfrak{N}}\xi(K)<\infty$ for any compact set $K\subset\mathbb{R}^d$, we see that the quasi-regularity of the Dirichlet form $(\mathfrak{E}, \mathfrak{D})$ implies that 
for any compact set $K$,  $\Xi_t(K)$, the number of particles of the process in $K$, will not diverge to infinity for any time $t\ge 0$, even though a jump rate is a long range.

These assumptions are satisfied for the system of interacting $\alpha$-stable process $(\alpha \in (0, 2))$, if $\alpha$ is strictly greater than $\kappa$, where $\kappa$ is the growth order of the density (the $1$-correlation function) of $\mu$, that is,
$\rho^1(x)=O(|x|^\kappa)$, $|x|\to\infty$. In particular, if $\mu$ is translation invariant, then $\kappa=0$ and the system can be constructed for any parameter $\alpha\in (0,2)$. The condition that $\alpha> \kappa$ seems to be best possible because it is a necessary condition to construct the independent system of infinite $\alpha$-stable processes. 

  %%%%%%%%%%%%%%%%%%%%%%%%%%%%%%%%%%%%%%%%%%%%%%%%%%%%%%%%%%

This paper is organized as follows: In Section 2 we introduce some notations and state our main results, Theorems \ref{Theorem.1} and \ref{CQG} in this paper. We give applications of theorems in Section 3. We prove Theorem \ref{CQG} in Section 4 and Theorem \ref{Theorem.1} in Section 5.

%%%%%%%%%%%%%%%%%%%%%%%%%%%%%%%%%%%%%%%%%%%%%%%%%%%%%%%%%%
%%%  SEC2   %%%%%%%%%%%%%%%%%%%%%%%%%%%%%%%%%%%%%%%%%%%%%%
%%%%%%%%%%%%%%%%%%%%%%%%%%%%%%%%%%%%%%%%%%%%%%%%%%%%%%%%%%
\section{Setup and main results} %%%%%%%%%%%%%%%%%%%%%%%%%
%%%%%%%%%%%%%%%%%%%%%%%%%%%%%%%%%%%%%%%%%%%%%%%%%%%%%%%%%%

Let $S$ be a closed set in $\R^d$ such that $0 \in S$ and $\overline{S^{{\rm{int}}}}=S$, where $S^{{\rm{int}}}$ denotes the interior of $S$. 
Let $\frakM$ be the configuration space over $S$ defined by
$$
\frakM=\frakM(S)=\{ \xi = \sum_i \delta_{x_i}; \xi(K)<\infty \mbox{ for all compact sets } K \subset S\}
$$ 
where $\delta_a$ stands for the delta measure at $a$.  $\frakM$ is a Polish space with the vague topology.  A probability measure $\mu$ on $\frakM$ is called a random point field. 
%Firstly we introduce bilinear form to describe our infinite particle system.

For $\sfn \in \{ 0 \} \cup \mathbb{N} \cup \{ \infty \}$ we put $\frakM_{\sfn} = \{ \xi \in \frakM ; \xi(S)=\sfn \}$ and introduce a map $\x_{\sfn} = (x_{\sfn}^1, x_{\sfn}^2, \ldots, x_{\sfn}^{\sfn}) ; \frakM_{\sfn} \to S^{\sfn}$ such that $\xi = \sum_{j=1}^{\sfn} \delta_{x_{\sfn}^j(\xi)}$. The map $\x_{\sfn}$ is called an $S^{\sfn}$-coordinate of $\xi$. We put $U_r = \{ x \in S ; |x| \leq r \}$ and
\begin{equation*}
\frakM_{r, \sfn} = \{ \xi \in \frakM ; \xi(U_r)={\sfn} \}. 
\end{equation*}
Note that $\frakM = \sum_{{\sfn}=0}^{\infty} \frakM_{r, \sfn}$. 
We define $\pi_r : \frakM \to \frakM$ by $\pi_r(\xi) = \xi(\cdot \cap U_r)$, and $\pi_r^c : \frakM \to \frakM$ by $\pi_r^c(\xi) = \xi( \cdot \cap \{ S \setminus U_r \})$.
A function $\x_{r, \sfn} : \frakM_{r, \sfn} \to U_r^{\sfn}$ is called a $U_r^{\sfn}$-coordinate (or a coordinate on $\frakM_{r, \sfn}$) of $\xi$ if 
\begin{equation*}
\pi_r(\xi) = \sum_{j=1}^{\sfn} \delta_{x_{r, \sfn}^j(\xi)}, \quad \x_{r, \sfn}(\xi)=(x_{r, \sfn}^1(\xi), \ldots, x_{r, \sfn}^{\sfn}(\xi)). %\label{0.4}
\end{equation*}
 For $f : \frakM \to \mathbb{R}$ a function $f_{r, \xi}^{\sfn}(x) : \frakM \times U_r^{\sfn} \to \mathbb{R}$ is called the $U_r^{\sfn}$-representation of $f$ if $f_{r, \xi}^{\sfn}$ satisfies the following : 
\begin{enumerate}
\item[(1)] $f_{r, \xi}^{\sfn}(\cdot)$ is a permutation invariant function on $U_r^{\sfn}$ for each $\xi \in \frakM$. 
\item[(2)] $f_{r, \xi_{(1)}}^{\sfn}(\cdot) = f_{r, \xi_{(2)}}^{\sfn}(\cdot)$ if $\pi_r^c(\xi_{(1)}) = \pi_r^c(\xi_{(2)})$, $\xi_{(1)}, \xi_{(2)} \in \frakM_{r, \sfn}$. 
\item[(3)] $f_{r, \xi}^{\sfn}(\x_{r, \sfn}(\xi)) = f(\xi)$ for $\xi \in \frakM_{r, \sfn}$, where $\x_{r, \sfn}(\xi)$ is a $U_r^{\sfn}$-coordinate of $\xi$. 
\item[(4)] $f_{r, \xi}^{\sfn}(\cdot) = 0$ for $\xi \notin \frakM_{r, \sfn}$. 
\end{enumerate}
Note that $f_{r, \xi}^{\sfn}$ is uniquely determined and $f(\xi) = \sum_{{\sfn}=0}^\infty f_{r, \xi}^{\sfn}({\x}_{r, \sfn}(\xi))$. 
When $f$ is $\sigma[\pi_r]$-measurable, $U_r^{\sfn}$-representations are independent of $\xi$, and is denoted by $f_r^{\sfn}$ instead of $f_{r, \xi}^{\sfn}$. Let $\mathfrak{B}_r = \{ f : \frakM \to \mathbb{R} ; \text{ $f$ is $\sigma[\pi_r]$-measurable} \}$ and $\mathfrak{B}_r^{{\rm bdd}} = \{ f \in \mathfrak{B}_r ; \text{$f$ is bounded} \}$. 
We set 
\begin{equation*}
\mathfrak{B}_{\infty} = \bigcup_{r=1}^{\infty} \mathfrak{B}_r, \quad \quad 
\mathfrak{B}_{\infty}^{{\rm bdd}} = \bigcup_{r=1}^{\infty} \mathfrak{B}_r^{{\rm bdd}}, %\label{1.1}
\end{equation*}
and call a function in $\mathfrak{B}_{\infty}$ a {\it local} function.
Then, we introduce the set of all {\it local smooth} functions on $\frakM$ given by
\begin{equation}
\label{D0}
\Dc = \{ f \in \mathfrak{B}_{\infty} ; f_{r, \xi}^{\sfn} \; \text{is smooth on} \; U_r^{\sfn} \; \text{for each} \; {\sfn}, r, \xi \}.
\end{equation}
Note that for any $f\in\Dc$ we can find $r\in \N$ such that 
$f(\xi)=\sum_{n=0}^{\infty}f_r^{\sfn}({\x}_{r, \sfn}(\xi))$. A local smooth function is continuous with the vague topology:
%\begin{equation}
$\Dc \subset C(\frakM)$. %\label{1.3}
%\end{equation}

For measurable functions $f^{\sfn}, g^{\sfn}$ on $S^{\sfn}$ we put  
\begin{equation*}
D^{\sfn}[f^{\sfn}, g^{\sfn}](\x_{\sfn}) = \frac{1}{2}\sum_{j=1}^{\sfn} \int_{S} \nadeffsa{j}{y}f^{\sfn}(\x_{\sfn}) \nadeffsa{j}{y}g^{\sfn}(\x_{\sfn}) \nu( \x_{\sfn}, x_j ; y)dy, %,
\end{equation*}
where  
\begin{equation}\label{nabla}
\nadeffsa{j }{y} f^{\sfn}(\x_{\sfn}) = f^{\sfn}(x_1, \ldots, x_{j-1}, y, x_{j+1}, \ldots, x_{\sfn}) - f^{\sfn}(\x_{\sfn}), 
\quad j=1,2,\dots, \sfn,
\end{equation}
and $\nu( \x_{\sfn}, x_j ; y)$ is a nonnegative measurable function on $S^{\sfn}\times S \times S$ which is symmetric in $\x_{\sfn}$ and satisfies
\begin{equation*}%\label{c_nu}
\int_S (1\wedge |y-x_j|^2)\nu (\x_{\sfn}, x_j;y)dy <\infty,
\quad \x_{\sfn} \in S^{\sfn}, \; j=1,2,\dots,\sfn.
\end{equation*}
We often write $\nu(\xi, x ; y)$ for $\nu(\x_{\sfn}(\xi), x ; y)$ in case $\xi \in \frakM_{\sfn}$.
Let $f, g \in \Dc$ with 
$f(\xi)=\sum_{n=0}^{\infty}f_r^{\sfn}({\x}_{r, \sfn}(\xi))$ and $g(\xi)=\sum_{n=0}^{\infty}g_r^{\sfn}({\x}_{r, \sfn}(\xi))$.
We remark that although $f_r^{\sfn}$ and $g_r^{\sfn}$ are functions of $U_r^{\sfn}$, they are naturally extended to the functions $f^{\sf{n}}$ and $g^{\sf{n}}$ on $S^n$ satisfying
$$
f^{\sf{n}}({\x}_{\sf{n}})=f_r^{\sf{m}}(\hat{\x}_{\sf{m}}),
\quad
g^{\sf{n}}({\x}_{\sf{n}})=g_r^{\sf{m}}(\hat{\x}_{\sf{m}}), \quad
\mbox{ if }\sum_{i:x_i\in U_r}\delta_{x_i}= \sum_{i:\hat{x}_i\in U_r}\delta_{\hat{x}_i}.
$$
Then we introduce the square field defined by
%set $\D[f, g] : \frakM \to \mathbb{R}$ by 
$$
\D[f, g](\xi)=\lim_{\ell\to\infty} \D[f, g](\pi_\ell(\xi)) 
$$
with
\begin{equation*}
\D[f, g](\pi_\ell(\xi)) = \begin{cases} D^{\sfn}[f^{\sfn}, g^{\sfn}](\x_{\ell, \sfn}(\xi)) &\text{for $\xi \in \frakM_{\ell,\sfn}, \sfn \in \N$}, \\
0 &\text{for $\xi \in \frakM_{\ell,0}$}, \end{cases}
\end{equation*}
and the bilinear form defined by
\begin{equation}
\label{Dinfty}
\begin{split}
\mathfrak{E}(f, g) &= \int_{\frakM} \D[f, g](\xi)\mu(d\xi), \\ 
\Dinf &= \{ f \in \Dc \cap L^2(\frakM, \mu) ; \mathfrak{E}(f, f) < \infty \}. %\label{0.3}
\end{split}
\end{equation}

%%%%%%correlation%%%%%%%%%%%%%%%%%%%%%%%%%%%%%%%%%%%%%%%%%%%%%%%%%%%%%%%%%%%%%%%%%%%%%%%%
We say a nonnegative permutation invariant function $\rho^{\sfn}$ on $S^{\sfn}$ is the $\sfn$-correlation function of $\mu$ if 
\begin{equation*}
\label{corrdef} \int_{A_1^{k_1} \times \cdots \times A_m^{k_m}} \rho^{\sfn}(\x_{\sfn})d\x_{\sfn} = \int_{\frakM} \prod_{i=1}^m \frac{\xi(A_i)!}{(\xi(A_i)-k_i)!}\mu(d\xi)
\end{equation*}
for any sequence of disjoint bounded measurable subsets $A_1, \ldots, A_m \subset S$ and a sequence of natural numbers $k_1, \ldots, k_m$ satisfying $k_1+\cdots+k_m=\sfn$. 

Permutation invariant functions $\sigma_r^{\sfn} : U_r^{\sfn} \to [0,\infty)$ are called density functions of $\mu$ if
\begin{equation*}
\frac{1}{{\sfn}!}\int_{U_r^{\sfn}}f_r^{\sfn}(\x_{\sfn})\sigma_r^{\sfn}(\x_{\sfn})d\x_{\sfn} = \int_{\frakM_{r, \sfn}}f(\xi)\mu(d\xi) \; \text{for all bounded $\sigma[\pi_r]$-measurable functions $f$}. \label{0.5}
\end{equation*}
%Here $f_r^{\sfn} : U_r^{\sfn} \to \mathbb{R}$ is the permutation invariant function such that $f_r^{\sfn}(\x(\xi)) = f(\xi)$ for $\xi \in \frakM_{r,\sfn}$, where $\x$ is a $U_r^{\sfn}$-coordinate. 
%%%%%%%%%%%%%%%%%%%%%%%%%%%%%%%%%%%%%%%%%%%%%%%%%%%%%%%%%%%%%%%%%%%%%%%%%%%%%%%%%%%%%%%%%%

We introduce conditions (A.1)--(A.2):
\begin{enumerate}
\item[(A.1)] $(\mathfrak{E}, \Dinf)$ is closable on $L^2(\frakM, \mu)$, 
\item[(A.2)] $\sigma_r^k \in L^{p}(U_r^k, dx)$ for all $k, r \in \N$ with some $1<p \leq \infty$. 
\end{enumerate}
Under condition (A.1) we denote the closure of $(\mathfrak{E}, \Dinf)$ by $(\mathfrak{E}, \mathfrak{D})$. 

We also introduce conditions (B.0)--(B.4): 
\begin{enumerate}
\item[(B.0)] There exists a function $p(r)$ on $(0, \infty)$ such that $\nu(\xi, x; y) \leq C_1 p(|x-y|)$ for $\mu$-a.s. $\xi \in \frakM$ and $dx$-a.e. $x, y \in S$. 
\item[(B.1)] $\rho^1(x) = O\left( |x|^{\kappa} \right)$ as $|x| \to \infty$ for some $\kappa \geq 0$. 
\item[(B.2)] $p(r) = O(r^{-(d+\alpha)})$ as $r \to \infty$ for some $\alpha > \kappa$. 
\item[(B.3)] $p(r) = O(r^{-(d+\beta)})$ as $r \to +0$ for some $0 < \beta < 2$. 
\item[(B.4)] $\displaystyle \dfrac{{\rm Var}\left[ \xi(U_r) \right]}{\left( \E \left[ \xi(U_r) \right] \right)^2} = O\left( r^{-\delta} \right)$ as $r \to \infty$ for some $\delta>0$.  
\end{enumerate}

We remark that the LHS of (B.4) is rewritten by the $1$ and $2$-correlation functions of $\mu$ as follows: 
\begin{equation*}
\dfrac{{\rm Var}\left[ \xi(U_r) \right]}{\left( \E \left[ \xi(U_r) \right] \right)^2} = \frac{\int_{U_r} \rho^1(x)dx - \int_{U_r^2} \left( \rho^1(x_1)\rho^1(x_2) - \rho^2(x_1, x_2) \right) dx_1dx_2}{\left( \int_{U_r} \rho^1(x) dx \right)^2} . 
\end{equation*}
From the above expression we can readily check that (B.4) holds if $\mu$ is a Poisson random point field  or a determinantal random point field.

Now we state the main theorem.
Please refer to \cite{FOT,MR} for the definition of the quasi-regularity.  

%%%%%%%%%%%%%%%%%%%%%%
%%%%%%%%%%%%%%%%%%%%%%
\begin{theorem} \label{Theorem.1}
Suppose that (A.1)--(A.2), (B.0)--(B.4) hold. Then $(\mathfrak{E}, \mathfrak{D})$ is a quasi-regular Dirichlet form on $L^2(\frakM, \mu)$. 
\end{theorem}
 
By virtue of \cite[Theorem IV.3.5 and Theorem IV.5.1]{MR} we get the following: 
\begin{cor}
Suppose that (A.1)--(A.2), (B.0)--(B.4) hold. Then there exists a special standard process $(\Xi_t, \{ \mathbb{P}_{\xi} \}_{\xi \in \frakM})$ associated with 
$((\mathfrak{E}, \mathfrak{D}), L^2(\frakM, \mu))$. 
Moreover the process is reversible with respect to the measure $\mu$. 
\end{cor}

\begin{remark}
(i) \quad A function $F$ on $\frakM$ is called a polynomial function if $F$ is written as 
\begin{equation*}
F(\xi) = Q(\langle \phi_1, \xi \rangle, \langle \phi_2, \xi \rangle, \ldots, \langle \phi_{\ell}, \xi \rangle)
\end{equation*}
with $\phi_k \in C_0^{\infty}(S)$ and a polynomial function $Q$ on $\R^{\ell}$, where $\langle \phi, \xi \rangle = \int_S \phi(x) \xi(dx)$ and $C_0^{\infty}(S)$ is the set of smooth functions with compact support. 
We denote the set of all polynomial function on $\frakM$ by $\mathfrak{A}$.
By the same argument in \cite{OsaTanCore} we see that Theorem \ref{Theorem.1} holds if we replace $\Dc$ by $\mathfrak{A}$ in the definition of $\mathcal{D}_\infty$. 

\noindent (ii) \quad Conditions (B.1)--(B.3) imply that there exists a constant $R>0$ such that
\begin{equation}
\int_{S} dx \rho^1(x) \int_A p(|x-y|) dy \leq R \int_A \rho^1(y) dy, \label{rhojump}
\end{equation}
for any compact subset $A$. The property (\ref{rhojump}) is necessary for constructing the system of independent particles of jump type. We expect that we can replace condition (B.2) to (\ref{rhojump}) 
in Theorem  \ref{Theorem.1}.
\end{remark} 

We introduce a Hamiltonian on a bounded Borel set $A$ as follows: 
For Borel measurable functions $\Phi: S \to \mathbb{R} \cup \{ \infty \}$ and $\Psi: S \times S \to \R \cup \{ \infty \}$ with $\Psi(x, y) = \Psi(y, x)$, let
\begin{equation*}
\mathcal{H}_A^{\Phi, \Psi}(\x) = \sum_{x_i \in A}\Phi(x_i)+\sum_{x_i, x_j \in A, i<j} \Psi(x_i, x_j), \quad \text{where} \; (x_1, \ldots, x_{\sfn}) \in S^{\sfn}.
\end{equation*}
We assume $\Phi<\infty$ a.e. to avoid triviality. For two measures $\nu_1$, $\nu_2$ on a measurable space $(\Omega, \mathcal{B})$ we write $\nu_1 \leq \nu_2$ if $\nu_1(A) \leq \nu_2(A)$ for all $A \in \mathcal{B}$. Suppose that $\Omega$ is a topological space and $\mathcal{B}$ is the topological Borel field. We say a sequence of finite Radon measures $\{ \nu^N \}$ on $\Omega$ convergence weakly to a finite Radon measure $\nu$ if $\lim_{N \to \infty} \int fd\nu^N = \int fd\nu$ for any bounded continuous function $f$ on $\Omega$. 
Then we give the definition of quasi-Gibbs measures \cite{Osa13a, Osa13b}. 

\begin{defi} \label{def 5.1}
A probability measure $\mu$ is said to be a $(\Phi, \Psi)$-quasi Gibbs measure if there exists an increasing sequence $\{ b_r \}$ of natural numbers and measures $\{ \mu_{r, k}^\sfn \}$ such that, for each $r, \sfn \in \mathbb{N}$, $\mu_{r, k}^\sfn$ and $\mu_r^\sfn := \mu(\cdot \cap \frakM_{b_r, \sfn})$ satisfy
\begin{equation*}
\mu_{r, k}^\sfn \leq \mu_{r, k+1}^\sfn \; \text{for all $k$}, \quad \lim_{k \to \infty} \mu_{r, k}^\sfn = \mu_r^\sfn \quad \text{weakly}, 
\end{equation*}
and that, for all $r, \sfn, k \in \mathbb{N}$ and for $\mu_{r, k}^\sfn$-a.e. $\xi \in \frakM$, 
\begin{equation} \label{5.3}
C_2^{-1} e^{-\mathcal{H}_r(\x)}\1 ( \x \in \frakM_{b_r, \sfn}) \Lambda(d\x) \leq \mu_{r, k, \xi}^\sfn(d\x) \leq C_2 e^{-\mathcal{H}_r(\x)}\1 (\x \in \frakM_{b_r, \sfn}) \Lambda(d\x). 
\end{equation}
Here we set 
\begin{equation*}
\1(\ast) = \begin{cases} 1 &\text{if the proposition $\ast$ is correct,} \\ 0 &\text{otherwise}, \end{cases} 
\end{equation*}
$\mathcal{H}_r(\x) = \mathcal{H}_{U_{b_r}}^{\Phi, \Psi}(\x)$, $C_2$ is a positive constant depending on $r, \sfn, k, \pi_{U_{b_r}^c}(\xi)$, $\Lambda$ is the Poisson random point field whose intensity is the Lebesgue measure on $S$, and $\mu_{r, k, \xi}^\sfn$ is the conditional probability measure of $\mu_{r, k}^\sfn$ defined by
\begin{equation} \label{5.4}
\mu_{r, k, \xi}^\sfn(d\x)=\mu_{r, k}^\sfn(\pi_{U_{b_r}} \in d\x|\pi_{U_{b_r}^c}(\xi)). 
\end{equation}
\end{defi} 

We remark that $(\Phi, \Psi)$-canonical Gibbs measures are $(\Phi, \Psi)$-quasi Gibbs measures. 
The converse is not always true. For example, Sine random point field, Ginibre random point field and Airy random point field are not canonical Gibbs measures but quasi Gibbs measures (\cite{Osa13a,Osa13b}). We introduce some assumptions. 
\begin{enumerate}
\item[(A.3)] $\mu$ is a $(\Phi, \Psi)$-quasi Gibbs measure. 
\item[(A.4)] There exist upper semi-continuous functions $\Phi_0: S \to \R \cup \{ \infty \}$, $\Psi_0: S \times S \to \R \cup \{ \infty \}$ and positive constants $C_3$ and $C_4$ such that 
\begin{equation*}
C_3^{-1}\Phi_0(s) \leq \Phi(s) \leq C_3\Phi_0(s), \quad \text{for all $s \in \R^d$, } %s \in \R^d
\end{equation*}
\begin{equation*}
C_4^{-1}\Psi_0(s, t) \leq \Psi(s, t) \leq C_4\Psi_0(s, t), \quad \text{for all $s, t \in \R^d$}. 
\end{equation*}
%Moreover, $\Phi_0$ and $\Psi_0$ are locally bounded from below and $\Gamma := \{ s; \Psi_0(s)=\infty \}$ is a compact set. 
\end{enumerate}
Then we give a sufficient condition for closability.

\begin{theorem} \label{CQG}
Assume (A.3) and (A.4). Then $(\mathfrak{E}, \Dinf)$ is closable. 
\end{theorem}

%%%%%%%%%%%%%%%%%%%%%%%%%%%%%%%%%%%%%%%%%%%%%%%%%%%%%%%%%%
%%%  SEC3   %%%%%%%%%%%%%%%%%%%%%%%%%%%%%%%%%%%%%%%%%%%%%%
%%%%%%%%%%%%%%%%%%%%%%%%%%%%%%%%%%%%%%%%%%%%%%%%%%%%%%%%%%
\section{Applications} %%%%%%%%%%%%%%%%%%%%%%%%%%%%%%%%%%%%%%%%
%%%%%%%%%%%%%%%%%%%%%%%%%%%%%%%%%%%%%%%%%%%%%%%%%%%%%%%%%%

In this section we give some applications to Theorem  \ref{Theorem.1}. 

\subsection{Interacting L\'{e}vy processes}

In this subsection we set $S = \R^d$. From conditions (B.2) and (B.3) we can construct infinite particle systems of L\'{e}vy processes with interaction. In particular, if we take 
$$
\nu(\xi, x; y) = c(d, \alpha)^{-1} |x-y|^{-d-\alpha}
$$ for $\alpha > \kappa$ with the constant $c(d, \alpha)$ given by 
\begin{equation*}
c(d, \alpha) = \frac{2^{-\alpha+1}\pi^{\frac{d+1}{d}}}{\Gamma\left( \frac{\alpha}{2}+1 \right) \Gamma\left( \frac{a+d}{2} \right)\sin\frac{\pi\alpha}{2}}, 
\end{equation*}
then we can construct infinite particle systems of (symmetric) $\alpha$-stable process with interaction. 
In this situation the bilinear form is given by  
\begin{equation*}
\mathfrak{E}_{\alpha}(f, g) = \frac{1}{2} \int_{\frakM} \mu(d\xi) \int_{\R^d} \xi(dx) \int_{\R^d} \frac{\{ f(\xi^{x y})-f(\xi) \} \{ g(\xi^{x y})-g(\xi) \}}{c(d, \alpha)|x-y|^{d+\alpha}} dy. 
\end{equation*}
From assumption (B.2) if $\mu$ is a translation invariant measure like ${\rm{Sine}}_{\beta}$ random point field with $\beta=1,2,4$ or Ginibre random point field, then we can construct interacting (symmetric) $\alpha$-stable processes for all $0<\alpha<2$. 

Moreover, we can apply our theorem to interacting (symmetric) stable-like processes. For a measurable function $\alpha : \R^d \to \R$ we define the following bilinear form :
%$(\mathfrak{E}_{{\rm SL}, \alpha}, \mathcal{C}^{{\rm lip}}_0(\R^d))$ :
\begin{equation*}
\tilde{\mathfrak{E}}_{{\rm SL}, \alpha}(u, v) = \frac{1}{2} \iint_{\R^d \times \R^d - \Delta} \frac{(u(x)-u(y))(v(x)-v(y))}{|x-y|^{d+\alpha(x)}}dx dy, \; u,v \in \mathcal{C}^{{\rm lip}}_0(\R^d),
\end{equation*}
where $\Delta = \{ (x, x) ; x \in \R^d \}$ and and $\mathcal{C}^{{\rm lip}}_0(\R^d)$ is the space of all uniformly Lipschitz continuous functions with compact support.
Under these assumptions \\
(i) \quad $0 < \alpha(x) < 2$, $dx$-a.e. $x \in \R^d$, \\
(ii) \quad $\frac{1}{2-\alpha}$, $\frac{1}{\alpha} \in L^{1}_{\rm loc}(\R^d)$, \\
(iii) \quad there exists compact set $K \in \R^d$ such that $\int_{\R^d \setminus K} |x|^{-d-\alpha(x)}dx < \infty$. \\
 It is known that the bilinear form is closable and the closure $(\tilde{\mathfrak{E}}_{{\rm SL}, \alpha}, \mathcal{F}_{{\rm SL}, \alpha})$ is a regular Dirichlet form on $L^2(\R^d, dx)$.
Thus there exists a Hunt process $(X_t^{{\rm SL}, \alpha}, \P_x^{{\rm SL}, \alpha})$ associated with $(\tilde{\mathfrak{E}}_{{\rm SL}, \alpha}, \mathcal{F}_{{\rm SL}, \alpha})$, which is a (symmetric) stable-like process \cite{Uem02}.  
Let $\nu(\xi, x; y) = |x-y|^{-d-\alpha(x)}$ such that $\alpha(x)$ satisfies the above conditions (ii), (iii) and $\kappa < \alpha(x) < 2$, $dx$-a.e. $x \in \R^d$. Then we can construct infinite systems of (symmetric) stable-like process with interaction. In this situation, the bilinear form is given by  
\begin{equation*}
\mathfrak{E}_{{\rm SL}, \alpha}(f, g) = \frac{1}{2} \int_{\frakM} \mu(d\xi) \int_{\R^d} \xi(dx) \int_{\R^d \setminus \{ x \}} \frac{\{ f(\xi^{x y})-f(\xi) \} \{ g(\xi^{x y})-g(\xi) \}}{|x-y|^{d+\alpha(x)}} dy. 
\end{equation*}

\subsection{Examples of generator and ISDE for the processes}

In this subsection we give examples of the $L^2$-generator for the processes. 
\begin{ex}
(i) \quad Let $\mu$ be a Gibbs measure with a self potential $\Phi$ and a interaction potential $\Psi$. Then we have 
\begin{equation*}
\frac{\rho^1(y)}{\rho^1(x)}\frac{d\mu_y}{d\mu_x}(\xi \setminus x) = \exp\left\{ -\Phi (y)+\Phi (x)-\sum_i \left\{ \Psi(x_i, y) - \Psi(x_i, x) \right\} \right\}, 
\end{equation*}
for $x \in \xi$ and $\xi \setminus x = \sum_i \delta_{x_i}$. 
Then the $L^2$-generator associated with the process $\Xi_t$ is the closure of the operator
\begin{equation*}
L f(\xi) = \int_{\R^d} \xi(dx) \int_{\R^d} dy \; c(\xi, x; y)[f(\xi^{x y})-f(\xi)], %\label{Liggen}
\end{equation*}
with 
$$
c(\xi, x;y)=\nu(\xi,x;y)+ \nu(\xi^{x y},y;x)
\exp\left\{ -\Phi (y)+\Phi (x)-\sum_i \left\{ \Psi(x_i, y) - \Psi(x_i, x) \right\} \right\}.
$$
In particular, when $\nu(\xi^{x y},x;y)=p(|x-y|)$ and $\Phi=0$,
$$
c(\xi, x;y)=p(|x-y|) \left\{ 1+ 
\prod_{i}\frac{\exp\{\Psi(x_i, x)\}}{\exp\{\Psi(x_i, y)\}}
\right\}.
$$
Here we give two examples of interaction potentials with long range. 

\noindent (1) (Lennard-Jones 6-12 potential)
%$\mu_{{\rm LJ612}}$ 
$$
\Psi(x, y)=|x-y|^{-12}-|x-y|^{-6}, \quad x,y \in \R^3.
$$
 
\noindent (2) (Riesz potential) \quad Let $a>d$.
$$
\Psi(x, y)=|x-y|^{-a}, \quad x,y \in \R^d.
$$  

\noindent
(ii) \quad Our result is more interesting for quasi-Gibbs measures which are not Gibbs measures. 
The determinatal point fields, Sine$_\beta$, Bessel$_{\alpha,\beta}$, Airy$_\beta$ and Ginibre random point fields are examples of them.
The interaction potentials of Sine$_\beta$, Bessel$_{\alpha,\beta}$ and Airy$_\beta$ random point fields
are given by
$$
\Psi(x,y)=\beta\log |x-y|, \quad x,y\in\R.
$$
Ginibre random random point field $\mu_{{\rm gin}}$ is a probability measure on the configuration space on $\R^2$ with self potential $\Phi(x)=0$ and interaction potential $\Psi(x, y) = -2\log|x-y|$. From Theorem 1.3 in Osada and Shirai \cite{OsaShi} we see that %$c_{\xi}(x, y)$ in (\ref{cxy}) is written by
\begin{equation*} 
c(\xi, x; y) = \nu(\xi,x;y)+ \nu(\xi^{x y},y;x)
\lim_{r \to \infty} \prod_{|x_i|<r} \frac{|y-x_i|^2}{|x-x_i|^2}. 
\end{equation*}
\end{ex}
In the forthcoming paper \cite{EsaSDE} we discuss that the associated labeled process solves the following ISDE:  
\begin{equation*}
X_j(t) = X_j(0) + \int_{[0, t) \times \R^d \times [0, \infty)}N(dsdudr) u \1\left( 0 \leq r \leq c(\sfX(s-), X_j(s-), X_j(s-)+u) \right),
\end{equation*}
for $j \in \N$, where $\sfX(t) = \sum_i \delta_{X_i(t)}$ and $N(dsdudr)$ is a Poisson random point field on $[0, \infty) \times \R^d \times [0, \infty)$ with intensity $dsdudr$.

\subsection{Comments for Glauber dynamics}

Let $\mu$ be a random point field and $\mu_x$, $x\in S$ be its Palm measures.
Suppose that $\mu_x$ is absolutely continuous with respect to $\mu$. Then we can consider the operator $L_{\rm Gla}$ on $L^2(\mu)$ defined by 
\begin{equation*}
L_{\rm Gla} f(\xi) = \int_S (f(\xi \cdot x)-f(\xi))\rho(x) \frac{d\mu_x}{d\mu}(\xi)dx + \int_{S} \xi(dx) (f(\xi \setminus x)-f(\xi)). 
\end{equation*}
Here we set $\xi \cdot x = \xi + \delta_{x}$ for $\xi \in \frakM$ and $x \in S$. 
The associated process can be regarded as the continuum version of Glauber dynamics. This is associated with the bilinear form 
\begin{equation*}
\mathfrak{E}_{\rm Gla} (f, g) = \int_{\frakM} \mu(d\xi) \int_{S} \xi(dx) (f(\xi \setminus x) - f(\xi))(g(\xi \setminus x) - g(\xi)). 
\end{equation*}
Under the same assumptions on $\mu$ as in Theorem \ref{Theorem.1}, we can show that the closure of $(\mathfrak{E}_{\rm Gla}, \Dinf)$ is a quasi-regular Dirichlet form by the same argument. Whereas if $\mu_x$ is singular to $\mu$ such as Ginibre random point field, the operator $L_{\rm Gla}$ can not be defined and the associated Glauber dynamics could not exist. 

%%%%%%%%%%%%%%%%%%%%%%%%%%%%%%%%%%%%%%%%%%%%%%%%%%%%%%%%%%
%%%  SEC4   %%%%%%%%%%%%%%%%%%%%%%%%%%%%%%%%%%%%%%%%%%%%%%
%%%%%%%%%%%%%%%%%%%%%%%%%%%%%%%%%%%%%%%%%%%%%%%%%%%%%%%%%%
\section{Proof of Theorem \ref{CQG}} %%%%%%%%%%%%%%%%%%%%%%%%%%%%%%%%%%%%%%
%%%%%%%%%%%%%%%%%%%%%%%%%%%%%%%%%%%%%%%%%%%%%%%%%%%%%%%%%%

In this section we prove Theorem \ref{CQG}. 
%In this definition $\mathbb{D}_r^{\sfn, i}[f, g]$, $i=1, 2, 3$ are well-defined, that is, these don't depend on how to chose $U_{b_r}^{\sfn}$-coordinate $\x_{r, \sfn}(\xi)$. 
%Let $\mu_r^\mathsf{n}$ and $\mu_{r, k}^{\sfn}$ be as in Definition \ref{def 5.1} and $\mu_{r, k, \xi}^{\sfn}$ be as in (\ref{5.4}). 
%Then we introduce the bilinear forms on $\Dc$ defined as 
%\begin{align*}
%\mathfrak{E}_r(f, g) &:= \int_{\frakM} \mathbb{D}_r[f, g](\xi)\mu(d\xi), &\quad \mathfrak{E}_r^{\sfn}(f, g) &:= \int_{\frakM} \mathbb{D}_r^{\sfn}[f, g](\xi)\mu_r^{\sfn}(d\xi), \\ %\label{1.4}
%\mathfrak{E}_{r, k}^{\sfn}(f, g) &:= \int_{\frakM} \mathbb{D}_r^{\sfn}[f, g]\mu_{r, k}^{\sfn}(d\xi), &\quad \mathfrak{E}_{r, k, \xi}^{\sfn}(f, g) &:= \int_{\frakM} \mathbb{D}_r^{\sfn}[f, g]\mu_{r, k, \xi}^{\sfn}(d\xi)
%\end{align*}
%
%Since $\sum_{\sfn=0}^{\infty} \mu_r^{\sfn} = \mu$, we see that
%$\mathfrak{E}_r(f, g) = \sum_{\sfn=0}^{\infty} \mathfrak{E}_r^{\sfn}(f, g)$. 
%In addition we have 
%\begin{equation*}
%\mathfrak{E}_{r, k}^{\sfn}(f, g) = \int_{\frakM}\mathfrak{E}_{r, k, \xi}^{\sfn}(f, g)\mu_{r, k}^{\sfn}(d\xi), 
%\end{equation*}
%\begin{equation*}
%\|f\|^2_{L^2(\frakM_{r, \sfn}, \mu_{r, k}^{\sfn})} = \int_{\frakM} \|f\|^2_{L^2(\frakM_{r, \sfn}, \mu_{r, k, \xi}^{\sfn})}\mu_{r, k}^{\sfn}(d\xi). 
%\end{equation*}
As a first step, we prepare some notations. For $\sfm \ge \sfn$, $\ell \ge r$, $k\in\N$ and $\xi\in\frakM$,
we introduce the measure $\munm$ on $U_{b_r}^{\sf{n}} \times (U_{b_\ell}\setminus U_{b_r})^{\sf{m}-\sf{n}}$ determined by 
$$
\int_{\frakM_{r,\sf{n}} \cap \frakM_{\ell,\sf{m}}}f(\eta)\mu_{\ell,k,\xi}^{\sf{m}}(d\eta
)=\int_{U_{b_r}^{\sf{n}} \times (U_{b_\ell}\setminus U_{b_r})^{\sf{m}-\sf{n}}}f_{\ell,\xi}^{\sf{m}}(\x_{\sfm})
\munm(d\x_{\sfm}),
$$
for any bounded measurable function $f$ on $\frakM$.
By (\ref{5.3}), for $\mu_{\ell,k}^{\sfm}$-a.s. $\xi\in\frakM$,
$\munm$ has a density $\chesig_{r,\ell, k, \xi}^{\sfn,\sfm}(\x_{\sfm})$ with respect to $e^{-\mathcal{H}_\ell(\x_{\sfm})}d\x_{\sfm}$. Here $d\x_{\sfm}$ denotes the Lebesgue measure on $U_{b_\ell}^{\sfm}$. 
%and we regard $e^{-\mathcal{H}_\ell(\x_{\sfm})}$ as a symmetric function on $U_{b_\ell}^{\sfm}$. 
We note that according to (\ref{5.3}), $\chesig_{r,\ell, k, \xi}^{\sfn,\sfm}(\x_{\sfm})$ is uniformly positive and bounded on $U_{b_r}^{\sf{n}} \times (U_{b_\ell}\setminus U_{b_r})^{\sf{m}-\sf{n}}$. 
We consider a bilinear form $\Eq$. For $q \in \N$ put $\nu_q(\xi, x; y)=\nu(\xi, x; y)\1 (\nu(\xi, x; y)\le q )$
and 
\begin{align*}
&\Lambda_q(\x_{\sfm})= \prod_{i=1}^{\sfm}\1 
\Big( |\Phi(x_i)+\sum_{\substack{j \neq i \\ 1 \leq j \leq \sfm}} \Psi(x_i, x_j)|\le q \Big), 
\\
&\Lambda_q(y|\x_{\sfm}) = \1\Big( |\Phi(y)+\sum_{j =1}^{\sfm} \Psi(y, x_j)|\le q \Big). 
\end{align*}
It is readily seen that
$\{ \nu_q \}$ is nondecreasing sequence in $q$ such that $\lim_{q\to\infty}\nu_q(\xi, x; y) = \nu(\xi, x; y)$ for $\mu$-a.s. $\xi \in \frakM$ and $dx$-a.e. $x, y \in S$, and $\{ \Lambda_q(\x_{\sfm}) \}$, $\{ \Lambda_q(y|\x_{\sfm}) \}$ are nondecreasing sequences in $q$. 
We put $\Eq = \Eqban{1}+\Eqban{2}+\Eqban{3}$, 
where 
\begin{align*}
\Eqban{1}(f, g) &= \int_{U_{b_r}^{\sf{n}} \times (U_{b_\ell}\setminus U_{b_r})^{\sfm-\sfn}} \munm(d\x_{\sfm}) \Lambda_q(\x_{\sfm})  \\
&\quad \times \sum_{j=1}^{\sfn} \int_{U_{b_r}} \Lambda_q(y|\x_{\sfm}^{\hiku{j}}) \nadeffsa{j}{y}f_{\ell, \xi}^{\sfm}(\x_{\sfm}) \cdot \nadeffsa{j}{y}g_{\ell, \xi}^{\sfm}(\x_{\sfm}) \nu_q(\xi, x_j; y) dy, \\
\Eqban{2}(f, g) &= \int_{U_{b_r}^{\sf{n}} \times (U_{b_\ell}\setminus U_{b_r})^{\sfm-\sfn}} \munm(d\x_{\sfm}) \Lambda_q(\x_{\sfm})  \\
&\quad \times \sum_{j=1}^{\sfn} \int_{U_{b_\ell}\setminus U_{b_r}} \Lambda_q(y|\x_{\sfm}^{\hiku{j}}) \nadeffsa{j}{y}f_{\ell, \xi}^{\sfm}(\x_{\sfm}) \cdot \nadeffsa{j}{y}g_{\ell, \xi}^{\sfm}(\x_{\sfm}) \nu_q(\xi, x_j; y) dy, \\
\Eqban{3}(f, g) &= \int_{U_{b_r}^{\sf{n}} \times (U_{b_\ell}\setminus U_{b_r})^{\sfm-\sfn}} \munm(d\x_{\sfm}) \Lambda_q(\x_{\sfm})   \notag \\
&\quad \times \sum_{j=\sfn+1}^{\sfm}
\int_{U_{b_r}} \Lambda_q(y|\x_{\sfm}^{\hiku{j}}) \nadeffsa{j}{y}f_{\ell, \xi}^{\sfm}(\x_{\sfm}) \cdot \nadeffsa{j}{y}g_{\ell, \xi}^{\sfm}(\x_{\sfm}) \nu_q(\xi, x_j; y)dy.
\end{align*}
Here for $(f_{r, \xi}^{\sf{m}})_{\sf{m}=\{0\}\cup\N}$, which are $U_{b_r}$-representations of $f$, we set 
\begin{align*}
\nadeffsa{j}{y}f_{r, \xi}^{\sfn}(\x_{\sfn}) &= \begin{cases} f_{r, \xi \cdot y}^{{\sfn}-1}(\x_{\sfn}^{\hiku{j}}) - f_{r, \xi}^{\sfn}(\x_{\sfn}), \quad \text{if $\x_{\sfn} \in U_{b_r}^{\sfn}$, $y \notin U_{b_r}$, }
\\
f_{r, \xi}^{{\sfn}}(\x_{\sfn}^{\hiku{j}}\cdot y) - f_{r, \xi}^{\sfn}(\x_{\sfn}), \quad \text{if $\x_{\sfn} \in U_{b_r}^{\sfn}$, $y \in U_{b_r}$, }
\end{cases}
\end{align*}
for $j=1,2,\dots, \sfn,$ with
$\x_{\sfn}^{\hiku{j}}= (x_{\sfn}^1, \ldots, x_{\sfn}^{j-1}, x_{\sfn}^{j+1}, \ldots, x_{\sfn}^{\sfn}) \in S^{\sfn -1}$ 
and 
$\x_{\sfn} \cdot y = (x_{\sfn}^1, \ldots, x_{\sfn}^{\sfn}, y) \in S^{\sfn +1}$. 

Then we show the following lemma. 

\begin{lemma} \label{lem 5.4}
Assume (A.3) and (A.4). Then $(\mathfrak{E}_{r, \ell, q, k, \xi}^{\sfn, \sfm}, \Dinf)$ is closable on $L^2(\frakM_{q, \sfm}, \mu_{q, k, \xi}^{\sfm})$ for $\sfm \ge \sfn$, $\ell \ge r$, $k, q\in\N$ and $\mu_{\ell, k}^{\sfm}$-a.s. $\xi$. 
\end{lemma}
\begin{proof}
By the simple observation we see
\begin{equation} \label{uekara}
\Eq(f, f) \leq  \mathfrak{E}_{\ell, \ell, q, k, \xi}^{\sfm, \sfm, 1}(f,f) \equiv \Eqm(f, f) \quad \text{for any $f \in \Dinf$}. 
\end{equation}
Setting $\chemu_{\ell, \ell, k, \xi}^{\sfm, \sfm} \equiv \chemu_{\ell, k, \xi}^{\sfm}$,  
we have
\begin{align*}
&\Eqm(f, f) 
\\ \notag
&= \int_{U_{b_\ell}^{\sfm}} \chemu_{\ell, k, \xi}^{\sfm}(d\x_{\sfm}) \Lambda_q(\x_{\sfm}) \sum_{j=1}^{\sfm} \int_{U_{b_\ell}} \Lambda_q(y|\x_{\sfm}^{\hiku{j}})(f_{\ell, \xi}^{\sfm}(\x_{\sfm}^{x_j y}) - f_{\ell, \xi}^{\sfm}(\x_{\sfm}))^2 \nu_q(\xi, x_j; y) dy \\
&\leq 2\int_{U_{b_\ell}^{\sfm}} \chemu_{\ell, k, \xi}^{\sfm}(d\x_{\sfm}) \Lambda_q(\x_{\sfm}) \sum_{j=1}^{\sfm} \int_{U_{b_{\ell}}} \Lambda_q(y|\x_{\sfm}^{\hiku{j}})\{ f_{\ell, \xi}^{\sfm}(\x_{\sfm}^{x_j y})^2 + f_{\ell, \xi}^{\sfm}(\x_{\sfm})^2 \} \nu_q(\xi, x_j; y) dy \notag \\
&\leq 2q\int_{U_{b_\ell}^{\sfm}} \chemu_{\ell, k, \xi}^{\sfm}(d\x_{\sfm}) \Lambda_q(\x_{\sfm}) \sum_{j=1}^{\sfm} \int_{U_{b_{\ell}}} \Lambda_q(y|\x_{\sfm}^{\hiku{j}}) f_{\ell, \xi}^{\sfm}(\x_{m}^{\hiku{j}} \cdot y)^2 dy \notag \\
&\quad + 2\sfm q|U_{b_\ell}| \|f\|_{L^2(\frakM_{\ell, \sfm}, \mu_{\ell, k, \xi}^{\sfm})}^2 \notag 
\end{align*}
and the first term of the right hand side is bounded by  
\begin{align*}
& 2qe^{2q}\sum_{j=1}^{\sfm} \int_{U_{b_{\ell}}^{\sfm+1}} \Lambda_q(\x_{\sfm})\chesig_{\ell, k, \xi}^{\sfm}(\x_{\sfm}) \frac{\chesig_{\ell, k, \xi}^{\sfm}(\x_{\sfm}^{\hiku{j}} \cdot y)}{\chesig_{\ell, k, \xi}^{\sfm}(\x_{\sfm}^{\hiku{j}} \cdot y)} \Lambda_q(y|\x_{\sfm}^{\hiku{j}}) e^{-\mathcal{H}_{\ell}(\x_{\sfm}^{\hiku{j}} \cdot y)} f_{\ell, \xi}^{\sfm}(\x_{\sfm}^{\hiku{j}} \cdot y)^2 d\x_{\sfm}dy \notag \\
&\leq 2C_2^2e^{2q}\sfm q|U_{b_\ell}|\|f\|_{L^2(\frakM_{\ell, \sfm}, \mu_{\ell, k, \xi}^{\sfm})}^2, \notag
\end{align*}
where $C_2$ is the constant in (\ref{5.3}). 
Hence we have 
\begin{equation} \label{iishiki}
\Eqm(f, f) \leq 2(C_2^2e^{2q}+1)\sfm q|U_{b_\ell}|\|f\|_{L^2(\frakM_{\ell, \sfm}, \mu_{\ell, k, \xi}^{\sfm})}^2. 
\end{equation}
Let $\{ f_n \}$ be a $\Eq$-Cauchy sequence in $\Dinf$ such that $\displaystyle \lim_{n \to \infty} \|f_n\|_{L^2(\frakM_{\ell, \sfm}, \mu_{\ell, k, \xi}^{\sfm})}=0$. 
Then by (\ref{uekara}), (\ref{iishiki}) we get $\displaystyle \lim_{n \to \infty}\Eq(f_n, f_n) =0$ for  $\mu_{\ell, k}^{\sfm}$-a.s. $\xi$. Hence, $(\Eq, \Dinf)$ is closable on $L^2(\frakM_{\ell, \sfm}, \mu_{\ell, k, \xi}^{\sfm})$. Thus we complete the proof.
\end{proof}

We use the following lemma, which is given in \cite[Lemma 2.1(1),(2)]{O96}.  See also \cite[Prop. I.3.7]{MR}. 

\begin{lemma} \label{O96lem2.1.1} {\rm (\cite[Lemma 2.1(1),(2)]{O96})} 
Let $\{ (\mathfrak{E}^{(n)}, \mathfrak{D}^{(n)}) \}_{n \in \N}$ be a sequence of positive definite, symmetric bilinear forms on $L^2(\frakM, \mu)$. \\
{\rm (1)} \  Suppose that $(\mathfrak{E}^{(n)}, \mathfrak{D}^{(n)})$ is closable for any $n \in \N$ and $\{ (\mathfrak{E}^{(n)}, \mathfrak{D}^{(n)}) \}$ is increasing, i.e. for any $n \in \N$ we have $\mathfrak{D}^{(n)} \supset \mathfrak{D}^{(n+1)}$ and $\mathfrak{E}^{(n)}(f, f) \leq \mathfrak{E}^{(n+1)}(f, f)$ for any $f \in \mathfrak{D}^{(n+1)}$. Let $\mathfrak{E}^{\infty}(f, f) = \lim_{n \to \infty} \mathfrak{E}^{(n)}(f, f)$ with the domain $\mathfrak{D}^{\infty} = \{ f \in \cap_{n \in \N} \mathfrak{D}^{(n)}; \sup_{n \in \N}\mathfrak{E}^{(n)}(f, f) < \infty \}$. Then $(\mathfrak{E}^{\infty}, \mathfrak{D}^{\infty})$ is closable on $L^2(\frakM, \mu)$. \\
{\rm (2)} \  In addition to the assumption (1), assume $(\mathfrak{E}^{(n)}, \mathfrak{D}^{(n)})$ are closed. Then $(\mathfrak{E}^{\infty}, \mathfrak{D}^{\infty})$ is closed. 
\end{lemma}

For $r, \ell, k, \sfn, \sfm \in \N$, $\xi \in \frakM$, we set 
$$\mathfrak{E}_{r, \ell, k, \xi}^{\sfn, \sfm}(f, f) = \lim_{q \to \infty} \mathfrak{E}_{r, \ell, q, k, \xi}^{\sfn, \sfm}(f, f), \; f \in \Dinf, 
$$ 
and put
$$
\mathfrak{E}_{r,\ell, k}^{\sfn, \sfm}(f, f) = \int_{\frakM} \lim_{q \to \infty} \mathfrak{E}_{r, \ell, q, k, \xi}^{\sfn, \sfm}(f, f) \mu_{\ell, k}^{\sfm}(d\xi), \; f \in \Dinf,
$$
where $\mu_{\ell, k, \xi}^{\sfm}(d\x_{\sfm})$ is defined in (\ref{5.4}).  %$= \mu_{\ell, k}^{\sfm}(\left. \pi_{U_{b_r}} \in d\x_{\sfm} \right| \pi_{U_{b_r}^c}(\xi))$.

\begin{lemma} \label{lem 5.3} %{\rm (\cite[Lemma 3.4]{Osa13a})} 
Assume that $(\mathfrak{E}_{r, \ell, k, \xi}^{\sfn, \sfm}, \Dinf)$ is closable on $L^2(\frakM_{\ell, \sfm}, \mu_{\ell, k, \xi}^{\sfm})$ for $\mu_{\ell, k}^{\sfm}$-a.s. $\xi$. Then $(\mathfrak{E}_{r, \ell, k}^{\sfn, \sfm}, \Dinf)$ is closable on $L^2(\frakM, \mu_{\ell, k}^{\sfm})$. 
\end{lemma}

\begin{proof}
Let $\{ f_{p} \}$ be a $\mathfrak{E}_{r, \ell, k}^{\sfn, \sfm}$-Cauchy sequence in $\Dinf$ such that $\lim_{p \to \infty} \| f_p \|_{L^2(\frakM, \mu_{\ell, k}^{\sfm})}=0$. By the definition of $\mu_{\ell, k, \xi}^{\sfm}$
\begin{equation*}
\| f \|_{L^2(\frakM_{\ell, \sfm}, \mu_{\ell, k}^{\sfm})}^2 = \int_{\frakM} \| f \|_{L^2(\frakM_{\ell, \sfm}, \mu_{\ell, k, \xi}^{\sfm})}^2 \mu_{\ell, k}^{\sfm}(d\xi). 
\end{equation*}
Hence we have 
\begin{align}
\lim_{p, q \to \infty} \int_{\frakM} \mathfrak{E}_{r, \ell, k, \xi}^{\sfn, \sfm}(f_p-f_q, f_p-f_q) \mu_{\ell, k}^{\sfm}(d\xi) &= 0 \label{ueshiki} \\
\lim_{p \to \infty} \int_{\frakM} \| f_p \|_{L^2(\frakM_{\ell, \sfm}, \mu_{\ell, k, \xi}^{\sfm})}^2 \mu_{\ell, k}^{\sfm}(d\xi) &= 0. \label{shitashiki}
\end{align}
By the definition of $\mathfrak{E}_{r, \ell, k, \xi}^{\sfn, \sfm}$ we have 
\begin{align}
&\int_{\frakM_{\ell, \sfm}} \mathfrak{E}_{r, \ell, k, \xi}^{\sfn, \sfm}(f_p-f_{p+1}, f_p-f_{p+1}) \mu_{\ell, k}^{\sfm}(d\xi) \label{L2kakikae}\\
&= \int_{\frakM_{\ell, \sfm}} \mu_{\ell, k}^{\sfm}(d\xi) \int_{U_{b_r}^{\sf{n}} \times (U_{b_\ell}\setminus U_{b_r})^{\sfm-\sfn}} \munm(d\x_{\sfm}) \notag \\
&\quad \times \int_{U_{b_{\ell}}} \sum_{j=1}^{\sfm} \{ \nadeffsa{j}{y}(f_p)_{\ell, \xi}^{\sfm}(\x_{\sfm}) - \nadeffsa{j}{y}(f_{p+1})_{\ell, \xi}^{\sfm}(\x_{\sfm}) \}^2 \notag \\
&\quad \times \1\{(x_j \in U_{b_r}, y \in U_{b_{\ell}}) \cup (x_j \in U_{b_\ell}\setminus U_{b_r}, y \in U_{b_r})\} \nu(\xi, x_j; y) dy. \notag
\end{align}
For a function $f$ on $\frakM_{\ell, \sfm}$, we set a vector valued function $(\hat{f}(\xi, \x_{\sfm}, y; j))_{j=1}^{\sfm}$ on $\frakM_{\ell, \sfm} \times U_{b_r}^{\sfn} \times U_{b_{\ell}}^{\sfm-\sfn} \times U_{b_{\ell}}$ defined by
\begin{equation*}
\hat{f}(\xi, \x_{\sfm}, y; j) = \nadeffsa{j}{y}f_{\ell, \xi}^{\sfm}(\x_{\sfm}) \1\{(x_j \in U_{b_r}, y \in U_{b_{\ell}}) \cup (x_j \in U_{b_\ell}\setminus U_{b_r}, y \in U_{b_r})\} \sqrt{\nu(\xi, x_j; y)}. 
\end{equation*}
Hence by (\ref{L2kakikae}) we have 
\begin{align*}
&\int_{\frakM_{\ell, \sfm}} \mathfrak{E}_{r, \ell, k, \xi}^{\sfn, \sfm}(f_p-f_{p+1}, f_p-f_{p+1}) \mu_{\ell, k}^{\sfm}(d\xi) \\
&= \int_{\frakM_{\ell, \sfm}} \mu_{\ell, k}^{\sfm}(d\xi) \int_{U_{b_r}^{\sf{n}} \times (U_{b_\ell}\setminus U_{b_r})^{\sfm-\sfn}} \munm(d\x_{\sfm}) \notag \\
&\quad \times \int_{U_{b_{\ell}}} \sum_{j=1}^{\sfm} \{ \hat{f_p}(\xi, \x_{\sfm}, y;j) - \hat{f_{p+1}}(\xi, \x_{\sfm}, y;j) \}^2 dy. \notag
\end{align*}
Combining this with (\ref{ueshiki}), we see that the sequence $\{ \hat{f_p} \}$ of functions on $\frakM_{\ell, \sfm} \times U_{b_r}^{\sf{n}} \times (U_{b_\ell}\setminus U_{b_r})^{\sfm-\sfn} \times U_{b_{\ell}}$ is a Cauchy sequence in $L^2(\frakM_{\ell, \sfm} \times U_{b_r}^{\sf{n}} \times (U_{b_\ell}\setminus U_{b_r})^{\sfm-\sfn} \times U_{b_{\ell}}, \mu_{\ell, k}^{\sfm}(d\xi) \otimes \munm(d\x_{\sfm}) \otimes dy)$. 
Therefore, there exists a vector valued function $h \in L^2(\frakM_{\ell, \sfm} \times U_{b_r}^{\sf{n}} \times (U_{b_\ell}\setminus U_{b_r})^{\sfm-\sfn} \times U_{b_{\ell}}, \mu_{\ell, k}^{\sfm}\otimes \munm \otimes dy)$ such that 
$$
\| \hat{f_p} - h \|_{L^2(\frakM_{\ell, \sfm} \times U_{b_r}^{\sf{n}} \times (U_{b_\ell}\setminus U_{b_r})^{\sfm-\sfn} \times U_{b_{\ell}}, \mu_{\ell, k}^{\sfm} \otimes \munm \otimes dy)} \to 0 \mbox{ as $p \to \infty$}. 
$$
Hence we can choose a subsequence $\{ f^{(1)}_p \}$ of $\{ f_p \}$ such that 
\begin{equation} \label{dai1noshiki}
\lim_{p \to \infty} \int_{U_{b_r}^{\sf{n}} \times (U_{b_\ell}\setminus U_{b_r})^{\sfm-\sfn}} \munm(d\x_{\sfm}) \int_{U_{b_{\ell}}} \sum_{j=1}^{\sfm} \{ \hat{f^{(1)}_p}(\xi, \x_{\sfm}, y;j) - h(\xi, \x_{\sfm}, y;j) \}^2 dy = 0,
\end{equation}
$\mu_{\ell, k}^{\sfm}$-a.s. $\xi$. 
By (\ref{ueshiki}) and (\ref{shitashiki}) we can choose a subsequence $\{ f^{(2)}_p \}$ of $\{ f^{(1)}_p \}$ such that 
\begin{align*}
\lim_{p, q \to \infty} \mathfrak{E}_{r, \ell, k, \xi}^{\sfn, \sfm}(f^{(2)}_p-f^{(2)}_q, f^{(2)}_p-f^{(2)}_q) &= 0, \quad \text{$\mu_{\ell, k}^{\sfm}$-a.s. $\xi$, } \\ %\label{asueshiki} \\
\lim_{p \to \infty} \| f^{(2)}_p \|_{L^2(\frakM_{\ell, \sfm}, \mu_{\ell, k, \xi}^{\sfm})}^2 &= 0, \quad \text{$\mu_{\ell, k}^{\sfm}$-a.s. $\xi$. } %\label{asshitashiki}
\end{align*}
%Let $\{ f^{(1)}_p \}$ be any subsequence of $\{ f_p \}$. Then by (\ref{ueshiki}) and (\ref{shitashiki}) we can choose a subsequence $\{ f^{(2)}_p \}$ such that $\mu_{\ell, k}^{\sfm}(\sfA_p) \leq 2^{-k}$ and $\mu_{\ell, k}^{\sfm}(\sfA_p) \leq 2^{-k}$, where
%\begin{align*}
%\sfA_q &= \{ \xi \in \frakM; \mathfrak{E}_{r, \ell, k, \xi}^{\sfn, \sfm}(f^{(2)}_p-f^{(2)}_{p+1}, f^{(2)}_p-f^{(2)}_{p+1}) \geq 2^{-2k} \}, \\
%\sfB_q &= \{ \xi \in \frakM; \| f^{(2)}_p \|_{L^2(\frakM_{\ell, \sfm}, \mu_{\ell, k}^{\sfm})} \geq 2^{-2k} \}. 
%\end{align*}
%Then from the Borel-Cantelli lemma, we see 
%\begin{equation*}
%\mu_{\ell, k}^{\sfm}(\limsup_{p \to \infty} \sfA_p) = \mu_{\ell, k}^{\sfm}(\limsup_{p \to \infty} \sfB_p) = 0. 
%\end{equation*}
%This means that for $\mu_{\ell, k}^{\sfm}$-a.s. $\xi$, the sequence $\{ f^{(2)}_p \}$ satisfies 
%\begin{equation*}
%\lim_{p, q \to \infty} \mathfrak{E}_{r, \ell, k, \xi}^{\sfn, \sfm}(f^{(2)}_p-f^{(2)}_q, f^{(2)}_p-f^{(2)}_q) = 0, \quad \lim_{p \to \infty} \| f^{(2)}_p \|_{L^2(\frakM_{\ell, \sfm}, \mu_{\ell, k}^{\sfm})} = 0, 
%\end{equation*}
%i.e. $\{ f^{(2)}_p \}$ is an $\mathfrak{E}_{r, \ell, k, \xi}^{\sfn, \sfm}$-Cauchy sequence conversing to $0$ in $L^2(\frakM_{\ell, \sfm}, \mu_{\ell, k, \xi}^{\sfm})$ as $p \to \infty$. 
Therefore, by assumption, we have 
\begin{equation} \label{f2pconv0}
\lim_{p \to \infty} \mathfrak{E}_{r, \ell, k, \xi}^{\sfn, \sfm}(f^{(2)}_p, f^{(2)}_p) = 0 \quad \text{for $\mu_{\ell, k}^{\sfm}$-a.s. $\xi$. }
\end{equation}
By (\ref{dai1noshiki}), (\ref{f2pconv0}) and the definition of $\mathfrak{E}_{r, \ell, k, \xi}^{\sfn, \sfm}$, we obtain $h = 0$  for $\mu_{\ell, k}^{\sfm} \otimes \munm \otimes dy$-a.e. $(\xi, \x_{\sfm}, y)$. Thus we have 
\begin{equation*}
\| \hat{f_p} \|_{L^2(\frakM_{\ell, \sfm} \times U_{b_r}^{\sf{n}} \times (U_{b_\ell}\setminus U_{b_r})^{\sfm-\sfn} \times U_{b_{\ell}},\mu_{\ell, k}^{\sfm} \otimes \munm \otimes dy)} \to 0 \quad \text{as $p \to \infty$.}
\end{equation*}
By the definition of $\mathfrak{E}_{r, \ell, k}^{\sfn, \sfm}$, this implies $\lim_{p \to \infty} \mathfrak{E}_{r, \ell, k}^{\sfn, \sfm}(f_p, f_p) = 0$. Thus the proof is completed. 
\end{proof}

%\begin{lemma} \label{lem 5.2} {\rm (\cite[Lemma 3.3]{Osa13a})} 
%Assume that $(\mathfrak{E}_k, \Dinf)$ is closable on $L^2(\frakM, \mu_{\ell, k}^{\sfm})$ for all $k$. Then $(\mathfrak{E}, \Dinf)$ is closable on $L^2(\frakM, \mu)$. 
%\end{lemma}

We set 
\begin{equation*}
\mathfrak{E}^{\infty}(f, f) = \lim_{r \to \infty} \lim_{\ell \to \infty} \sum_{\sfm=0}^{\infty} \sum_{\sfn=0}^{\sfm} \lim_{k \to \infty} \left( \int_{\frakM} \lim_{q \to \infty} \mathfrak{E}_{r, \ell, q, k, \xi}^{\sfn, \sfm}(f, f) \mu_{\ell, k}^{\sfm}(d\xi) \right) \quad \text{for any $f \in \Dinf$}. 
\end{equation*}
Note that $\Eq$ is increasing in $\sfn, \sfm, r, \ell, q, k$. 
Combining Lemma \ref{lem 5.3} with Lemma \ref{O96lem2.1.1} (1) and Lemma \ref{lem 5.4} we have the following lemma as a second step. 

\begin{lemma} \label{Ersfn}
Assume that (A.3) and (A.4) holds. Then $( \mathfrak{E}^{\infty}, \mathfrak{D}^{\infty})$ is closable on $L^2(\frakM, \mu)$. 
\end{lemma}

As a final step, we show Theorem \ref{CQG}. 
We show the following lemma.
\begin{lemma}\label{lem.2.2}
We have $\mathfrak{E}^{\infty}(f, f) = \mathfrak{E}(f, f) < \infty$ for all $f \in \Dinf$.
\end{lemma}
\begin{proof}
We consider the following square fields on $\Dinf$: for $f, g \in \Dinf$, $\xi\in \frakM_{r,\sfn}$, $r\in\N$, let
\begin{align*}
\mathbb{D}_r^{{\sfn}, 1}[f, g](\xi) &:= \sum_{j=1}^{\sfn} \int_{U_{b_r}} \nadeffsa{j}{y}f_{r, \xi}^{\sfn}(\x_{r, \sfn}(\xi)) \cdot \nadeffsa{j}{y}g_{r, \xi}^{\sfn}(\x_{r, \sfn}(\xi)) \nu(\xi, x_{r, \sfn}^j(\xi); y) dy,
\\
\mathbb{D}_r^{{\sfn}, 2}[f, g](\xi) &:= \sum_{j=1}^{\sfn} \int_{U_{b_r}^c} \nadeffsa{j}{y}f_{r, \xi}^{\sfn}(\x_{r, \sfn}(\xi)) \cdot \nadeffsa{j}{y}g_{r, \xi}^{\sfn}(\x_{r, \sfn}(\xi)) \nu(\xi, x_{r, \sfn }^j(\xi); y) dy, 
\end{align*}
and
\begin{equation*}
\mathbb{D}_r^{{\sfn}, 3}[f, g](\xi) := \int_{U_{b_r}^c} \xi(dx) \int_{U_{b_r}} \nadeffsa{x}{y} f_{r, \xi}^{\sfn}(\x_{r, \sfn}(\xi)) \cdot \nadeffsa{x}{y}g_{r, \xi}^{\sfn}(\x_{r, \sfn}(\xi)) \nu(\xi, x; y)dy,
\end{equation*}
where we set
\begin{equation*}
\nadeffsa{x}{y}f_{r, \xi}^{\sfn}(\x_{\sfn}) = f_{r, \xi \setminus x}^{{\sfn}+1}(\x_{\sfn} \cdot y) - f_{r, \xi}^{\sfn}(\x_{\sfn}), \quad \text{if $\xi_{U_{b_r}^c}(x) \geq 1$, $y \in U_{b_r}$}. 
\end{equation*}
We put 
\begin{equation*}
\mathbb{D}_r^{\sfn}[f, g](\xi) = \begin{cases} \sum_{i=1}^3 \mathbb{D}_r^{{\sfn}, i}[f, g](\xi)  & \xi \in \frakM_{r, \sfn}, \\
 0 & \xi \notin \frakM_{r, \sfn}. \end{cases}
\end{equation*}
and 
\begin{equation*}
\mathbb{D}_r[f, g](\xi) = \sum_{\sfn =0}^{\infty} \mathbb{D}_r^{\sfn}[f, g](\xi). 
\end{equation*} 
Then we consider the bilinear forms on $\Dinf$ defined as $\mathfrak{E}_r(f, g) := \int_{\frakM} \mathbb{D}_r[f, g](\xi)\mu(d\xi)$ for $f, g \in \Dinf$. In addition we set
\begin{equation*}
\mathfrak{E}^r(f, g) = \lim_{\ell \to \infty} \sum_{\sfm=0}^{\infty} \sum_{\sfn=0}^{\sfm} \lim_{k \to \infty} \left( \int_{\frakM} \lim_{q \to \infty} \mathfrak{E}_{r, \ell, q, k, \xi}^{\sfn, \sfm}(f, g) \mu_{\ell, k}^{\sfm}(d\xi) \right) \quad \text{for any $f, g \in \Dinf$}. 
\end{equation*}
Then we can see 
\begin{equation} \label{2tsuonaji}
\mathfrak{E}_r(f, f) = \mathfrak{E}^r(f, f) \quad \text{for any $f \in \Dinf$.}
\end{equation}
Let $f \in \Dinf \cap \mathfrak{B}_r^{{\rm bdd}}$. Then we have $\sum_{{\sfn}=0}^{\infty} \D_r^{\sfn}[f, f](\xi) = \sum_{{\sfn}=0}^{\infty} \D_{\ell}^{\sfn}[f, f](\xi)$ for all $r \leq \ell$ and $\mu$-a.s. $\xi$. Then we have $\mathfrak{E}_r(f, f) = \mathfrak{E}_\ell(f, f) = \mathfrak{E}(f, f)$. Combining this with (\ref{2tsuonaji}) we can have $\lim_{r \to \infty} \mathfrak{E}^r(f, f) = \mathfrak{E}(f, f)$. 
Thus the proof is completed. 
\end{proof}

From Lemma \ref{O96lem2.1.1} (1) we have that $(\mathfrak{E}^{\infty}, \mathfrak{D}^{\infty})$ is closable, where we set
$$ \mathfrak{E}^{\infty}(f, f) = \lim_{r \to \infty} \mathfrak{E}^{r}(f, f), \quad \mathfrak{D}^{\infty} = \left\{ f \in \Dinf; \sup_{r \in \N}\mathfrak{E}^r(f, f) < \infty \right\}. $$
Combining this with Lemma \ref{lem.2.2} we see that $\mathfrak{E}^{\infty}(f, f) = \mathfrak{E}(f, f)$ for $f \in \Dinf$. Thus $\mathfrak{D}^{\infty} = \Dinf$. Then we show that $(\mathfrak{E}, \Dinf)$ is closable. Thus the proof of Theorem \ref{CQG} is completed. 

\begin{remark}
When $(\mathfrak{E}_r, \Dinf)$ is closable, we denote its closure by $(\mathfrak{E}_r, \hat{\mathfrak{D}}_r)$,  for $r \in \N$. Since $\{(\mathfrak{E}_r, \hat{\mathfrak{D}}_r)\}_{r\in\N}$ is increasing,  $(\hat{\mathfrak{E}}_{\infty}, \hat{\mathfrak{D}}_{\infty})$ is closed by Lemma \ref{O96lem2.1.1} (2), where 
\begin{equation*}
\hat{\mathfrak{E}}_{\infty}(f, f) = \lim_{r \to \infty} \mathfrak{E}_r(f, f), \quad \hat{\mathfrak{D}}_{\infty} = \{ f \in \bigcap_{r \in \N} \hat{\mathfrak{D}}_r; \sup_{r \in \N}\mathfrak{E}_{r}(f, f) < \infty \}.
\end{equation*}
It is clear that $\mathfrak{D}_{\infty} \subset \hat{\mathfrak{D}}_{\infty}$ and $\hat{\mathfrak{E}}_{\infty}(f, f) = \mathfrak{E}(f, f)$ for $f \in \Dinf$. 
However $(\hat{\mathfrak{E}}_{\infty}, \hat{\mathfrak{D}}_{\infty})$ is not necessary to be the closure $(\mathfrak{E}, \mathfrak{D})$  of $(\mathfrak{E}, \Dinf)$.  $(\mathfrak{E}, \mathfrak{D})$  coincides with the decreasing limit of the closure of $(\mathfrak{E}_r, \Dinf \cap \mathfrak{B}_r^{{\rm bdd}})$ (See \cite[Lemma 2.1 (3), Lemma 2.2 (3)]{O96}). 
\end{remark}

%%%%%%%%%%%%%%%%%%%%%%%%%%%%%%%%%%%%%%%%%%%%%%%%%%%%%%%%%%
%%%  SEC6   %%%%%%%%%%%%%%%%%%%%%%%%%%%%%%%%%%%%%%%%%%%%%%
%%%%%%%%%%%%%%%%%%%%%%%%%%%%%%%%%%%%%%%%%%%%%%%%%%%%%%%%%%
\section{Proof of Theorem \ref{Theorem.1}}
%%%%%%%%%%%%%%%%%%%%%%%%%%%%%%%%%%%%%%%%%%%%%%%%%%%%%%%%%%

In this section we prove Theorem \ref{Theorem.1}.
Let $(\mathfrak{E}, \mathfrak{D})$ be the closure of $(\mathfrak{E}, \Dinf)$. We first examine the Markov property of the closed forms $(\mathfrak{E}, \mathfrak{D})$. 

\begin{lemma} 
$(\mathfrak{E}, \mathfrak{D})$ is Markovian. Then $(\mathfrak{E}, \mathfrak{D})$ is a Dirichlet form. 
\end{lemma}
\begin{proof}
For $\varepsilon > 0$ there exists $\varphi_{\varepsilon} \in C^{\infty}(\mathbb{R})$ such that $\varphi_{\varepsilon}(t) = t$ for all $t \in [0, 1]$, $\varphi_{\varepsilon}(t) \in [-\varepsilon, 1+\varepsilon]$ and $|\varphi_{\varepsilon}'(t)| \leq 1$ for all $t \in \mathbb{R}$. By the mean-value theorem we get 
\begin{align}
\D[\varphi_{\e} \circ f, \varphi_{\e} \circ f](\xi) &= \frac{1}{2} \sum_{j=1}^{\infty} \int_{S} \left\{ \nadeffsa{j}{y} ( \varphi_{\e} \circ f )(\xi) \right\}^2 \nu(\xi, x_j; y)dy \label{mvthm1} \\
&= \frac{1}{2} \sum_{j=1}^{\infty} \int_{S} \left\{ \varphi_{\e}'(c_{x_j, y}) \nadeffsa{j}{y} f(\xi) \right\}^2 \nu(\xi, x_j; y)dy, \notag
\end{align}
where $\xi = \sum_{j=1}^{\infty} \delta_{x_j}$ and $c_{x_j, y}$ is a constant depending on $x_j$ and $y$. 
Since $\sup_{t \in \R}|\varphi_{\e}'(t)| \leq 1$ holds, then we get  
\begin{align}
\text{RHS of (\ref{mvthm1})} &\leq \sup_{t \in \R}| \varphi_{\e}'(t)|^2 \cdot \frac{1}{2} \sum_{j=1}^{\infty} \int_{S} \left\{ \nadeffsa{j}{y} f(\xi) \right\}^2 \nu(\xi, x_j; y)dy \label{mvthm2} \\
&\leq \frac{1}{2} \sum_{j=1}^{\infty} \int_{S} \left\{ \nadeffsa{j}{y} f(\xi) \right\}^2 \nu(\xi, x_j; y)dy \leq \D[f, f](\xi). \notag
\end{align}
Then $\varphi_{\varepsilon} \circ f \in \Dinf$ for all $f \in \Dinf$ and from (\ref{mvthm2}) we get 
\begin{equation*}
\mathfrak{E}(\varphi_{\varepsilon} \circ f, \varphi_{\varepsilon} \circ f) 
= \int_{\frakM}  \D[\varphi_{\e} \circ f, \varphi_{\e} \circ f](\xi) \mu(d\xi)
\leq \int_{\frakM}  \D[f, f](\xi) \mu(d\xi) \\
= \mathfrak{E}(f, f). 
\end{equation*}
This implies $(\mathfrak{E}, \mathfrak{D})$ is Markovian (See \cite[Proposition I.4.10]{MR}). 
\end{proof}

We show the quasi-regularity of the Dirichlet form $(\mathfrak{E}, \mathfrak{D})$.
We introduce a mollifier on $\mathfrak{B}_{\infty}^{{\rm bdd}}$.  Let $\mathfrak{j} : \mathbb{R}^d \to \mathbb{R}$ be a non-negative, smooth function such that $\int_{\mathbb{R}^d} \mathfrak{j}(x)dx = 1$ and $\mathfrak{j}(x)=0$ for $|x| \geq \frac{1}{2}$. Let $\mathfrak{j}_{\e}(\cdot) = \e^d\mathfrak{j}(\cdot /\e )$ and $\mathfrak{j}_{\e}^{\sfn}(x_1, \ldots, x_{\sfn}) = \prod_{j=1}^{\sfn} \mathfrak{j}_{\e}(x_j)$. 
For $f \in \mathfrak{B}_r^{{\rm bdd}} \subset \mathfrak{B}_{r+\e}^{{\rm bdd}}$, we put 
\begin{equation*}
\hat{f}_{r+\e}^{\sfn}(\x_{\sfn}) = \mathfrak{j}_{\e}^{\sfn} \ast f_{r+\e, \xi}^{\sfn}(\x_{\sfn}) := \int_{(\R^d)^{\sfn}} \mathfrak{j}_{\e}^{\sfn}(\x_{\sfn}-y)f_{r+\e, \xi}^{\sfn}(y)dy, \quad \x_{\sfn} \in U_r^{\sfn}. 
\end{equation*}
Since $f$ is $\sigma[\pi_r]$-measurable, $f_{r+\e, \xi}^{\sfn}=f_{r+\e}^{\sfn}$ and $\hat{f}_{r+\e}^{\sfn}$ is a $U_{r+\e}$-representation of some $\sigma[\pi_{r+\e}]$-measurable function. We denote the function by $\mathfrak{J}_{r, \e}f$, that is, 
\begin{equation*}
\mathfrak{J}_{r, \e}f(\xi) = \hat{f}_{r+\e}^{\sfn}(\x), \quad \text{if $\xi \in \frakM_{\sfn, r+\e}$, $\sfn \in \N \cup \{ 0 \}$}. 
\end{equation*}

The following lemma is given in \cite[Lemma 2.4]{O96}. 

\begin{lemma}
\begin{enumerate}
\item Let $f \in \mathfrak{B}_{r}^{{\rm bdd}}$. Then we have the following: 
\begin{equation*}
\mathfrak{J}_{r, \e}f \in \Dinf \quad \text{for} \quad \e>0 , \quad \lim_{\e \to 0}\| \mathfrak{J}_{r, \e}f - f\|_{L^2(\frakM, \mu)} = 0. 
\end{equation*}
\item Let $f \in \mathfrak{B}_r^{{\rm bdd}}$ such that $f_r^{\sfn} \in C^{\infty}(U_r^{\sfn})$ for all ${\sfn}$. %, where $f_r^{\sfn}$ are $U_r^{\sfn}$-representations of $f$. 
Let $\delta > 0$ and ${\sf N}_{r, \delta} = \{ \xi \in \frakM; \xi(U_{r+\delta}-U_r)=0 \}$. Then 
\begin{equation*}
\lim_{\e \to 0} \mathfrak{J}_{r+\delta, \e}f(\xi) = f(\xi) \quad \text{for all $\xi \in {\sf N}_{r, \delta}$. }
\end{equation*}
\item $\Dinf$ is dense in $L^2(\frakM, \mu)$. 
\end{enumerate}
\end{lemma}

Let $\bd{A}=\{ \bda = \{ a_r \}_{r \in \N}; a_r \in \N, a_r \leq a_{r+1} \ \text{for all $r$} \}$. For $\bda= \{ a_r \}_{r \in \N} \in \bd{A}$, let $\frakM[\bda]= \{ \xi \in \frakM; \xi(\Qr) \leq a_r \ \text{for all $r$} \}$. Then $\frakM[\bda]$ is a compact set. (See \cite[Proposition 3.16]{Ru}. )
Suppose $\bda, \bdb \in \bd{A}$ and $c \in \R$.  we set $\bda_+ = \{ a_{r+1} \}_{r \in \N}$, $c\bda = \{ ca_r \}_{r \in \N}$ and $\bda + \bdb = \{ a_r + b_r \}_{r \in \N}$.  Let $\1$ be a sequence that $a_r =1$ for all $r \in \N$. 

\begin{lemma} \label{teiin}
Assume (B.4). Let $\bda_n = \{ a_{n, r} \}_{r \in \N} = \{ n2^{(d+\kappa)r} \}_{r \in \N} \in \bd{A}$, $n \in \N$. Then we have 
\begin{equation}
\mu\left(\bigcup_{n=1}^{\infty} \frakM[\bda_n] \right) = 1. \label{teiinosae}
\end{equation}
\end{lemma}
\begin{proof}
By Chebyshev's inequality we get 
\begin{align*}
\mu\left[ \xi(\Qr) > n\E^{\mu}(\xi(\Qr)) \right] &= \mu\left[ \xi(\Qr)-\E^{\mu}(\xi(\Qr)) > (n-1)\E^{\mu}(\xi(\Qr)) \right] \\
&\leq \frac{{\rm Var}(\xi(\Qr))}{(n-1)^2\E^{\mu}(\xi(\Qr))^2}. 
\end{align*}
Hence 
\begin{equation}
\sum_{n=2}^{\infty} \mu\left( \bigcup_{r=1}^{\infty} \bigg\{ \xi(\Qr) > n(\E^{\mu}(\xi(\Qr))) \bigg\} \right) \leq \sum_{n=2}^{\infty} \frac{1}{(n-1)^2} \sum_{r=1}^{\infty} \frac{{\rm Var}(\xi(\Qr))}{\E^{\mu}(\xi(\Qr))^2}. \label{borecansum} 
\end{equation}
By (B.4),  
\begin{equation*}
\frac{{\rm Var}(\xi(\Qr))}{\E^{\mu}(\xi(\Qr))^2} = \frac{{\rm Var}(\xi(U_{2^r}))}{\E^{\mu}(\xi(U_{2^r}))^2} = O(2^{-\delta r}) 
\end{equation*}
holds. Hence the RHS of (\ref{borecansum}) is finite. 
Therefore from Borel-Cantelli's Lemma, for $\mu$-a.s. $\xi$ there exists $n_0 \geq 2$ such that 
$\xi(\Qr) \leq n_0 \E^{\mu}(\xi(\Qr))$ for any $r\in\N$. 

By (B.1) we can check $\E^{\mu}[\xi(\Qr)] = O\left( 2^{(d+\kappa)r} \right)$.  Hence we see that for $\mu$-a.s. $\xi$, there exists $n \in \N$ such that $\xi(\Qr) \leq a_{n, r}$ for any $r \in \N$. This complete the proof. 
\end{proof}

We introduce the function $\chia$ defined by  
\begin{equation*}
\chi[\bda](\xi) = \rho \circ d_{\bda}(\xi), \quad d_{\bda}(\xi) = \sum_{r=1}^{\infty} \sum_{j \in J_{r, \xi}} \frac{(2^r-|x_j|) \wedge 2^{r-1}}{2^{r-1}a_r}, 
\end{equation*}
where $(x_j)_{j \in \N}$ is a sequence such that $|x_j| \leq |x_{j+1}|$ for all $j \in \N$,  
$\xi=\sum_{j} \delta_{x_j}$ and 
\begin{equation*}
J_{r, \xi}=\{ j; j>a_r, x_j \in \Qr \}. 
\end{equation*}
$\rho : \mathbb{R} \to [0, 1]$ is the function defined by 
\begin{equation*}
\rho(t) = \begin{cases}
1 &\text{if $t<0$}, \\
1-t &\text{if $0 \leq t \leq 1$},  \\
0 &\text{if $1<t$}. 
\end{cases}
\end{equation*}
 
\begin{lemma} \label{chiprop}
For any $\bda = \{ a_r \}_{r \in \N} \in \bd{A}$ we have 
\begin{equation*}
\chi[\bda](\xi)= 
\begin{cases}
1 &\text{if $\xi \in \frakM[\bda]$, } \\
\in [0, 1] &\text{if $\xi \in \frakM_{\bda}^{2\bda_+}$, } \\
0 &\text{if $\xi \in \frakM[2\bda_{+}]^c$, }
\end{cases}
\end{equation*}
where we set $\frakM_{\bda}^{\bdb} = \frakM[\bdb] \setminus \frakM[\bda]$ for $\bda, \bdb \in \bd{A}$. 
\end{lemma}
\begin{proof}
For $\xi \in \frakM[\bda]$, $J_{r, \xi} = \emptyset$ for all $r \in \N$. Then we can check $d_{\bda}(\xi) = 0$. Hence $\chi[\bda](\xi)=1$. 

If $\xi \in \frakM[2\bda_+]^c$ then there exists $r_0$ such that $\xi(U_{2^{r_0}}) \geq 2a_{r_0+1}+1$. Then 
\begin{equation}
\# J_{r_0+1, \xi} \geq a_{r_0+1}+1 \label{Jnumber}
\end{equation}
holds, where $\#J$ denotes the cardinality of the set $J$. We take $x_{j_0} \in \{ x_j \in U_{2^{r_0}}; j \in J_{r_0+1, \xi} \}$ and fix. Then $(2^{r_0+1}-|x_{j_0}|) \wedge 2^{r_0} = 2^{r_0}$ holds. Hence
\begin{equation} 
\frac{(2^{r_0+1}-|x_{j_0}|) \wedge 2^{r_0}}{2^{r_0}a_{r_0+1}} = \frac{1}{a_{r_0+1}}. \label{dhyouka}
\end{equation}
Therefore from (\ref{Jnumber}) and (\ref{dhyouka}), we have
\begin{equation*}
d_{\bda}(\xi) \geq \sum_{j \in J_{r_0+1, \xi}} \frac{(2^{r_0+1}-|x_j|) \wedge 2^{r_0}}{2^{r_0}a_{r_0+1}} \geq \frac{a_{r_0+1}+1}{a_{r_0+1}} \geq 1. 
\end{equation*}
Hence $\chi[\bda](\xi)=0$. 
Moreover since $\chia(\xi) = \rho \circ d_{\bda}(\xi)$ and $\rho(x) \in [0, 1]$ for all $x \in \R$, we see $\chia(\xi) \in [0, 1]$ for all $\xi \in \frakM$, especially for all $\xi \in \frakM_{\bda}^{2\bda_+}$. Thus the proof is completed. 
\end{proof}

From this lemma we can call $\chia$ a cut off function on $\frakM[\bda]$. 

\begin{lemma} \label{lem357}
Let $f \in \mathfrak{D}_{\infty}$ and $\bda=\{ a_r \} \in \bd{A}$ then we have 
\begin{equation}
\D[\chia f, \chia f](\xi) \leq 2 \left( \D[\chia , \chia](\xi)f(\xi)^2 +\D[f, f](\xi) \right),  \label{chiafchiaf} 
\end{equation}
and 
\begin{align}
&\D[(1-\chia )f, (1-\chia )f](\xi) \label{1-chichi} \\
&\leq \D[\chia , \chia ](\xi) f(\xi)^2 + \sum_{i=1}^{\infty} \int_{S} (1-\chia(\xi^{x_i, y}))^2 (f(\xi^{x_i, y})-f(\xi))^2 \nu(\xi, x_i; y)dy. \notag
\end{align}
\end{lemma}
\begin{proof}
Let $f, g \in \Dinf$ and $\xi=\sum_i \delta_{x_i}$. 
\begin{align}
&\D [fg, fg](\xi) = \frac{1}{2} \sum_{i=1}^{\infty} \int_{S} (f(\xi^{x_i, y})g(\xi^{x_i, y})-g(\xi)f(\xi))^2 \nu(\xi, x_i; y)dy \label{DfgDfg} \\
&= \frac{1}{2} \sum_{i=1}^{\infty} \int_{S} \{ (g(\xi^{x_i, y})-g(\xi)) f(\xi)+g(\xi^{x_i, y})(f(\xi^{x_i, y})-f(\xi)) \}^2 \nu(\xi, x_i; y)dy \notag \\
&\leq \sum_{i=1}^{\infty} \int_{S} \{ (g(\xi^{x_i, y})-g(\xi))^2 f(\xi)^2 + g(\xi^{x_i, y})^2(f(\xi^{x_i, y})-f(\xi))^2 \} \nu(\xi, x_i; y)dy. \notag
\end{align}
Substituting $\chia$ into $g$ in the equation (\ref{DfgDfg}), we have
\begin{align*}
\D [\chia f, \chia f](\xi) %\label{substichia} 
&\leq  \sum_{i=1}^{\infty} \int_{S} \{ (\chia (\xi^{x_i, y})- \chia (\xi))^2  \} \nu(\xi, x_i; y)dy f(\xi)^2 \\
&\quad + \sum_{i=1}^{\infty} \int_{S} \{ \chia^2(\xi^{x_i, y})(f(\xi^{x_i, y})-f(\xi))^2 \} \nu(\xi, x_i; y)dy \\
&\leq 2\left( \D[\chia , \chia](\xi) f(\xi)^2 + \D[f, f] \right), 
\end{align*}
since $\chia (\xi) \leq 1$. 
Hence we obtain (\ref{chiafchiaf}). 

We prove the second claim. 
Substituting $1-\chia$ into $g$ in the equation (\ref{DfgDfg}), we have
\begin{align}
&\D [(1-\chia)f, (1-\chia)f](\xi) 
\notag \\ 
&\leq    \sum_{i=1}^{\infty} \int_{S} ( \chia(\xi^{x_i, y})-\chia(\xi))^2  \nu(\xi, x_i; y)dy f(\xi)^2 \notag \\
&\hspace{66pt} +  \sum_{i=1}^{\infty} \int_{S} (1-\chia(\xi^{x_i, y}))^2 (f(\xi^{x_i, y})-f(\xi))^2  \nu(\xi, x_i; y)dy. \notag 
\end{align}
Hence we get (\ref{1-chichi}). Thus the proof is completed. 
\end{proof}

The next lemma is a key part of the proof of Theorem \ref{Theorem.1}. 
\begin{lemma} \label{lem358}
Assume {\rm (B.0)}, {\rm (B.2)} and {\rm (B.3)}. Let $\bda_n = \{ a_{n, r} \}_{r \in \N} = \{n2^{(d+\kappa)r} \}_{r \in \N}$. 
Then there exists $C=C_{d, \alpha, \beta, \kappa}$ such that
\begin{equation}
\int_{\frakM} \D[\chian , \chian](\xi) f(\xi)^2 \mu(d\xi) \leq C \int_{\sfAn} f(\xi)^2 \mu(d\xi),   \label{chiaC}
\end{equation}
for all $n \in \N$ and $f \in \Dinf$, where $\alpha > \kappa$, $0 < \beta <2$ are in {\rm (B.1)} and {\rm (B.2)}.
\end{lemma}
\begin{proof}
From assumption (B.0) we have
\begin{align} 
&\text{LHS of (\ref{chiaC})} \label{nup1}
\\ \notag
&\leq \int_{\sfAn} \mu(d\xi) f(\xi)^2 \cdot \frac{1}{2} \sum_{i=1}^{\infty} \int_S (\chian(\xi^{x_i, y})-\chian(\xi))^2 C_1p(|y-x_i|) dy. 
\end{align}
Here we used the fact that $\xi \notin {\sfAn}$ implies  $\chian(\xi^{x_i, y})-\chian(\xi) = 0$. From assumptions (B.2) and (B.3), it is enough to consider the case where 
\begin{equation*}
p(r) = 
\begin{cases}
r^{-d-\alpha} &\text{if $r \geq 1$}, \\ 
r^{-d-\beta} &\text{if $r < 1$}, 
\end{cases}
\end{equation*}
for $0<\kappa<\alpha$ and $0<\beta<2$. Since
\begin{equation*}
|\chia(\eta) - \chia(\xi)| \leq |d_{\bda}(\eta)-d_{\bda}(\xi)|, \quad \text{for $\bda \in \bd{A}$ and $\xi, \eta \in \frakM$, }
\end{equation*}
we have 
\begin{align}
&\text{RHS of (\ref{nup1})} \label{3581} 
\le \frac{C_1}{2}\int_{\sfAn} f(\xi)^2 \cdot (I_1(\xi)+I_2(\xi)) \mu(d\xi)
\end{align}
with
\begin{align}
%&\leq \int_{\sfAn} \mu(d\xi) f(\xi)^2 \cdot \frac{1}{2} \sum_{i=1}^{\infty} \int_{S} (d_{\bda_n}(\xi^{x_i, y})-d_{\bda_n}(\xi))^2 p(|y-x_i|) dy \notag \\
&I_1(\xi)= \sum_{i=1}^{\infty} \int_S (d_{\bda_n}(\xi^{x_i, y})-d_{\bda_n}(\xi))^2 \frac{\1( |y-x_i| < 1 )}{|y-x_i|^{d+\beta}} dy \notag \\
&I_2(\xi)= \sum_{i=1}^{\infty} \int_S (d_{\bda_n}(\xi^{x_i, y})-d_{\bda_n}(\xi))^2 \frac{\1(|y-x_i| \geq 1)}{|y-x_i|^{d+\alpha}} dy \notag 
\end{align}
Let $A_r= \Qr \setminus U_{2^{r-1}}$ for $r \geq 2$ and $A_1 = U_2$. We see that 
\begin{equation*}
|d_{\bda_n}(\xi^{x_i, y})-d_{\bda_n}(\xi)| \leq \big| |x_i|-|y| \big| \left( \frac{1}{2^{r-1}a_{n, r}}+\frac{1}{2^{r-2}a_{n, r-1}}+\frac{1}{2^{r}a_{n, r+1}} \right), 
\end{equation*}
for $r \in \N$, $x_i \in A_r$, $y \in A_{r-1} \cup A_r \cup A_{r+1}$ with $|y-x_i| < 1$. 
Combining this with 
\begin{equation*}
\frac{1}{2^{r-1}a_{n, r}}+\frac{1}{2^{r-2}a_{n, r-1}}+\frac{1}{2^{r}a_{n, r+1}} \leq \frac{3}{2^{r-2}a_{n, r-1}}, \quad \text{for $n, r \in \N$, }
\end{equation*}
we obtain
\begin{equation*}
|d_{\bda_n}(\xi^{x_i, y})-d_{\bda_n}(\xi)| \leq \big| |x_i|-|y| \big| \frac{3}{2^{r-2}a_{n, r-1}} \leq \big| y-x_i \big| \frac{3}{2^{r-2}a_{n, r-1}},   
\end{equation*}
for $r \in \N$, $x_i \in A_r$, $y \in A_{r-1} \cup A_r \cup A_{r+1}$ with $|y-x_i| < 1$. 
Hence we have 
\begin{align}
&I_1 (\xi)\leq  9 \sum_{r=1}^{\infty} \sum_{x_i \in A_r} \int_S \big| |x_i|-|y| \big|^2 \frac{1}{2^{2r-4}a_{n, r-1}^2}  \frac{\1( |y-x_i| < 1 )}{|y-x_i|^{d+\beta}} dy \label{3582} \\
&\leq 9 \sum_{r=1}^{\infty} \sum_{x_i \in A_r}  \frac{1}{2^{2r-4}a_{n, r-1}^2} \int_S \frac{\1( |y-x_i| < 1)}{|y-x_i|^{d+\beta-2}} dy \notag \\
&\leq  9\sum_{r=1}^{\infty} \frac{2a_{n, r+1}+1}{2^{2r-4}a_{n, r-1}^2} \cdot \int_S \frac{\1( |z| < 1 )}{|z|^{d+\beta-2}} dz, \mbox{ for $\xi\in \sfAn$} \notag
\end{align}
where we use $\xi(A_r) \leq 2a_{n, r+1}+1$ for $\xi \in \sfAn$. 
We remark that 
\begin{equation}
\int_S \frac{\1( |z| < 1 )}{|z|^{d+\beta-2}} dz  < \infty, \quad \text{for all $0<\beta<2$. } \label{C1}
\end{equation}
In addition we have
\begin{align}
\sum_{r=1}^{\infty} \frac{2a_{n, r+1}+1}{2^{2r-4}a_{n, r-1}^2} &= \sum_{r=1}^{\infty} \frac{n2^{(r+1)(d+\kappa)+1}+1}{2^{2r-4} \cdot n^22^{2(r-1)(d+\kappa)}} \label{C2} \\
&= \sum_{r=1}^{\infty} \frac{1}{n} \cdot \frac{1}{2^{r(d+\kappa+2)-3(d+\kappa)-5}} + \sum_{r=1}^{\infty} \frac{1}{n^2} \cdot \frac{1}{2^{2r(d+\kappa+1)-2(d+\kappa)-4}} \notag \\
&\leq \frac{(1/2)^{-2(d+\kappa)-3}}{1-(1/2)^{d+\kappa+2}} + \frac{(1/2)^{-2}}{1-(1/2)^{2(d+\kappa+1)}} <\infty. \notag
\end{align}
Then from From (\ref{3582}), (\ref{C1}) and (\ref{C2}), there exists a constant $C_5 = C_5(d, \beta, \kappa)>0$ %such that
%\begin{equation}
%\text{RHS of (\ref{3582})} \leq C_5. 
% \int_{\sfAn} f(\xi)^2 \mu(d\xi). 
%\label{3584}
%\end{equation}
% and (\ref{3584}) we then have 
\begin{equation}
I_1(\xi) \leq C_5, \mbox{ for $\xi\in \sfAn$}. %\int_{\sfAn} f(\xi)^2 \mu(d\xi). 
\label{I1bound}
\end{equation}
On the other hand we have
\begin{align}
&I_2(\xi) =  \sum_{r=1}^{\infty} \sum_{x_i \in A_r} \sum_{\ell =1}^{\infty} \int_{A_{\ell}}(d_{\bda_n}(\xi^{x_i, y})-d_{\bda_n}(\xi))^2 \frac{\1( |y-x_i| \geq 1 )}{|y-x_i|^{d+\alpha}} dy \notag \\ %\label{3585} \\
&=  \sum_{r=1}^{\infty} \sum_{x_i \in A_r} \int_{A_{r-1} \cup A_r \cup A_{r+1}}(d_{\bda_n}(\xi^{x_i, y})-d_{\bda_n}(\xi))^2 \frac{\1( |y-x_i| \geq 1 )}{|y-x_i|^{d+\alpha}} dy \notag \\ 
&\quad +  \sum_{r=1}^{\infty} \sum_{x_i \in A_r} \sum_{\substack{\ell =1 \\ |\ell-r|>1}}^{\infty} \int_{A_{\ell}}(d_{\bda_n}(\xi^{x_i, y})-d_{\bda_n}(\xi))^2 \frac{\1( |y-x_i| \geq 1 )}{|y-x_i|^{d+\alpha}} dy \notag \\
&\equiv I_3(\xi)+I_4(\xi). \notag
\end{align}
We see that 
\begin{equation*}
|d_{\bda_n}(\xi^{x_i, y})-d_{\bda_n}(\xi)| \leq \frac{1}{a_{n, r-1}}+\frac{1}{a_{n, r}}+\frac{1}{a_{n, r+1}}, 
\end{equation*}
for $r \in \N$, $x_i \in A_r$ and $y \in A_{r-1} \cup A_r \cup A_{r+1}$. 
Combining this with
\begin{equation*}
\frac{1}{a_{n, r}}+\frac{1}{a_{n, r-1}}+\frac{1}{a_{n, r+1}} \leq \frac{3}{a_{n, r-1}}, \quad \text{for $r, n \in \N$, }
\end{equation*}
we obtain 
\begin{equation*}
|d_{\bda_n}(\xi^{x_i, y})-d_{\bda_n}(\xi)| \leq \frac{3}{a_{n, r-1}}, 
\end{equation*}
for $r \in \N$, $x_i \in A_r$ and $y \in A_{r-1} \cup A_r \cup A_{r+1}$.
Hence we have 
\begin{align}
&I_3(\xi) \leq 9 \sum_{r=1}^{\infty} \sum_{x_i \in A_r} \int_{A_{r-1} \cup A_r \cup A_{r+1}} \frac{1}{a_{n, r-1}^2} \frac{\1( |y-x_i| \geq 1 )}{|y-x_i|^{d+\alpha}} dy \label{3586} \\ 
&\leq 9\sum_{r=1}^{\infty} \frac{2a_{n, r+1}+1}{a_{n, r-1}^2} \int_{A_{r-1} \cup A_r \cup A_{r+1}} \frac{\1( |y-x_i| \geq 1 )}{|y-x_i|^{d+\alpha}} dy \notag \\
&\leq 9\sum_{r=1}^{\infty} \frac{2a_{n, r+1}+1}{a_{n, r-1}^2} \int_{S} \frac{\1( |z| \geq 1 )}{|z|^{d+\alpha}} dz, 
 \mbox{ for $\xi\in \sfAn$}
\notag
\end{align} 
where we use $\xi(A_r) \leq 2a_{n, r+1}+1$ for $\xi \in \sfAn$. We remark that 
\begin{equation}
\int_{S} \frac{\1( |z| \geq 1 )}{|z|^{d+\alpha}} dz  < \infty, \quad \text{for all $\alpha>0$. } \label{C31}
\end{equation}
In addition we have
\begin{align}
\sum_{r=1}^{\infty} \frac{2a_{n, r+1}+1}{a_{n, r-1}^2} &= \sum_{r=1}^{\infty} \frac{n2^{(r+1)(d+\kappa)+1}+1}{n^22^{2(r-1)(d+\kappa)}} \label{C32} \\
&= \sum_{r=1}^{\infty} \frac{1}{n} \cdot \frac{1}{2^{r(d+\kappa)-3(d+\kappa)-1}} + \sum_{r=1}^{\infty} \frac{1}{n^2} \cdot \frac{1}{2^{2r(d+\kappa)-2(d+\kappa)}} \notag \\
&\leq \frac{(1/2)^{-2(d+\kappa)-1}}{1-(1/2)^{d+\kappa}} + \frac{1}{1-(1/2)^{2(d+\kappa)}} <\infty. \notag
\end{align}
Then from (\ref{3586}), (\ref{C31}) and (\ref{C32}), there exists a constant $C_6 = C_6(d, \alpha, \kappa)>0$ such that
%\begin{equation}
%\text{RHS of } = C_6 %\int_{\sfAn} f(\xi)^2\mu(d\xi). \label{3587}
%\end{equation}
%From (\ref{3586}) and (\ref{3587}) we then have
\begin{equation}
I_3(\xi) \leq C_6 %\int_{\sfAn} f(\xi)^2 \mu(d\xi)
, 
 \mbox{ for $\xi\in \sfAn$.}
 \label{I3bound}
\end{equation}
Moreover we see that  
\begin{equation*}
|d_{\bda_n}(\xi^{x_i, y})-d_{\bda_n}(\xi)| \leq \sum_{m=\ell \wedge r}^{\ell \vee r} \frac{1}{a_{n, m}},  
\end{equation*}
for $\ell, r \in \N$ with $|\ell-r|>1$, $x_i \in A_r$ and $y \in A_{\ell}$. 
Hence we have 
\begin{align}
&I_4(\xi) \leq  \sum_{r=1}^{\infty} \sum_{x_i \in A_r} \sum_{\substack{\ell =1 \\ |\ell-r|>1}}^{\infty} \int_{A_{\ell}}\left( \sum_{m=\ell \wedge r}^{\ell \vee r} \frac{1}{a_{n, m}} \right)^2 \frac{\1( |y-x_i| \geq 1 )}{|y-x_i|^{d+\alpha}} dy \label{3588} 
\\
&\leq \sum_{r=1}^{\infty} \sum_{x_i \in A_r} \sum_{\substack{\ell =1 \\ |\ell-r|>1}}^{\infty} \left( \sum_{m=\ell \wedge r}^{\ell \vee r} \frac{1}{a_{n, m}} \right)^2 \int_{A_{\ell}} \frac{\1( |y-x_i| \geq 1 )}{|2^{\ell \wedge r-1}-2^{\ell \vee r}|^{d+\alpha}} dy. \notag
\end{align}
We remark $|2^{\ell \wedge r-1}-2^{\ell \vee r}| \geq 2^{\ell \vee r-2}$ for any $\ell, r \in \N$ with $|\ell-r|>1$. Then 
\begin{align}
&\text{RHS of (\ref{3588})} \label{3589} \\
&\leq  \sum_{r=1}^{\infty} \sum_{x_i \in A_r} \sum_{\substack{\ell =1 \\ |\ell-r|>1}}^{\infty} \left( \sum_{m=\ell \wedge r}^{\ell \vee r} \frac{1}{a_{n, m}} \right)^2 \int_{A_{\ell}} \frac{\1( |y-x_i| \geq 1 )}{2^{(\ell \vee r-2)(d+\alpha)}} dy \notag 
\\
&\leq \sum_{r=1}^{\infty} \sum_{x_i \in A_r} \sum_{\substack{\ell =1 \\ |\ell-r|>1}}^{\infty} \left( \sum_{m=\ell \wedge r}^{\ell \vee r} \frac{1}{a_{n, m}} \right)^2 \frac{|A_{\ell}|}{2^{(\ell \vee r-2)(d+\alpha)}} \notag 
\\
&\leq \sum_{r=1}^{\infty} \sum_{\substack{\ell =1 \\ |\ell-r|>1}}^{\infty} \left( \sum_{m=\ell \wedge r}^{\ell \vee r} \frac{1}{a_{n, m}} \right)^2 \frac{(2a_{n, r+1}+1)|A_{\ell}|}{2^{(\ell \vee r-2)(d+\alpha)}}, \notag
\end{align}
where we use $\xi(A_r) \leq 2a_{n, r+1}+1$ for $\xi \in \sfAn$ and $|\cdot|$ denote the Lebesgue measure on $\R^d$. 
Since
\begin{align}
\sum_{m=\ell \wedge r}^{\ell \vee r} \frac{1}{a_{n, m}} &= \sum_{m=\ell \wedge r}^{\ell \vee r} \frac{1}{n2^{m(d+\kappa)}} %= \frac{1}{n} \cdot \frac{(1/2)^{(\ell \wedge r)(d+\kappa)} - (1/2)^{(\ell \vee r +1)(d+\kappa)}}{1-(1/2)^{d+\kappa}} \notag%\label{C41} 
%\\
\leq \frac{1}{n} \cdot \frac{(1/2)^{(\ell \wedge r)(d+\kappa)}}{1-(1/2)^{d+\kappa}}, \notag
\end{align}
we obtain 
\begin{align}
&\sum_{r=1}^{\infty} \sum_{\substack{\ell =1 \\ |\ell-r|>1}}^{\infty} \left( \sum_{m=\ell \wedge r}^{\ell \vee r} \frac{1}{a_{n, m}} \right)^2 \frac{(2a_{n, r+1}+1)|A_{\ell}|}{2^{(\ell \vee r-2)(d+\alpha)}} \label{C42} 
\\
&\leq \sum_{r=1}^{\infty} \sum_{\substack{\ell =1 \\ |\ell-r|>1}}^{\infty} \frac{1}{n^2}  \left\{ \frac{(1/2)^{(\ell \wedge r)(d+\kappa)}}{1-(1/2)^{d+\kappa}} \right\}^2 \frac{(n2^{(r+1)(d+\kappa)+1}+1) |A_{\ell}|}{2^{(\ell \vee r-2)(d+\alpha)}} \notag \\
&\leq C_7 \sum_{r=1}^{\infty} \sum_{\ell =r+1}^{\infty} \frac{1}{2^{r(d+\kappa)+\alpha\ell}} +C_8 \sum_{r=1}^{\infty} \sum_{\ell = 1}^{r-1} \frac{1}{2^{r(\alpha-\kappa)+\ell (d+2\kappa)}} < \infty, \notag
\end{align}
where $C_i = C_i(d, \alpha, \kappa)$ for $i=7, 8$ are constants depending on $d$, $\alpha$ and $\kappa$, and we use $\alpha-\kappa>0$. 
From (\ref{3588}), (\ref{3589}) and (\ref{C42}) there exists a constant $C_9 = C_9(d, \alpha, \kappa) >0$ such that
\begin{equation}
I_4(\xi) \leq C_9, \mbox{ for $\xi\in \sfAn$.} \label{I4bound}
\end{equation}
Therefore from (\ref{I1bound}), (\ref{I3bound}) and (\ref{I4bound}) we obtain
\begin{equation}
I_1(\xi) + I_2(\xi) = I_1(\xi)+I_3(\xi)+I_4(\xi) \leq C_5+C_6+C_9, \mbox{ for $\xi\in \sfAn$}
\label{I1+I2bound}
\end{equation}
and the lemma is proved with $C =\frac{C_1(C_5+C_6+C_9)}{2}$.\end{proof}

\begin{lemma} \label{chiadomain}
Let $\bda_n = \{ a_{n, r} \}_{r \in \N} = \{n2^{(d+\kappa)r} \}_{r \in \N}$. Then for $f \in \Dinf$ and $n \in \N$ we have 
\begin{equation}
\chian f \in \mathfrak{D} \label{chianfindomain}
\end{equation}
and  
\begin{equation}
\| (1-\chian)f \|_{1} \leq \sqrt{ (C+2) \int_{\frakM[\bda_n-\1]^c} \{ f(\xi)^2 +\D[f, f](\xi) \} \mu(d\xi)}, \label{kinjitane}
\end{equation}
where $C$ is the constant in Lemma \ref{lem358} and
$\|f\|_1^2 = \|f\|_{L^2(\frakM, \mu)}^2+\mathfrak{E}(f, f)$. 
\end{lemma}
\begin{proof}
By (\ref{chiafchiaf}) we have
\begin{align}
\mathfrak{E}(\chian f, \chian f) &\leq 2\int_{\frakM} \D[\chian , \chian](\xi)f(\xi)^2 \mu(d\xi) + 2\int_{\frakM} \D[f, f](\xi) \mu(d\xi) \label{Echiafchiaf} \\
&\leq 2\int_{\frakM} \D[\chian , \chian](\xi)f(\xi)^2 \mu(d\xi) + 4\mathfrak{E}(f, f). \notag
\end{align}
From Lemma \ref{lem358} the RHS of (\ref{Echiafchiaf}) is bounded by
\begin{equation*}
2C \int_{\sfAn} f(\xi)^2 \mu(d\xi) + 4\mathfrak{E}(f, f) \leq (2C + 4) \left( \|f\|_{L^2(\frakM, \mu)}^2 + \mathfrak{E}(f, f) \right) < \infty. 
\end{equation*}
Hence we obtain (\ref{chianfindomain}). 

We show the second claim. 
By (\ref{1-chichi}) we have
\begin{align}
&\mathfrak{E}((1-\chian)f, (1-\chian)f) \leq \int_{\frakM} \mu(d\xi) \D[\chian , \chian ](\xi) f(\xi)^2 \label{chiadomain1} \\
&\quad + \int_{\frakM} \mu(d\xi) \sum_{i}^{\infty} \int_{S} \{ (1-\chian(\xi^{x_i, y}))^2 (f(\xi^{x_i, y})-f(\xi))^2 \} \nu(\xi, x_i; y)dy. \notag
\end{align}
Since $1-\chian (\xi) = 0$ on $\frakM[\bda_n]$, from Lemma \ref{lem358} we have 
\begin{equation}
\text{RHS of (\ref{chiadomain1})} \leq C \int_{\frakM[\bda_n-\1]^c} f(\xi)^2 \mu(d\xi) + 2 \int_{\frakM[\bda_n-\1]^c} \D[f, f](\xi) \mu(d\xi). \label{chiadomain2}
\end{equation} 
Hence by (\ref{chiadomain2})
\begin{align}
\| (1-\chian)f \|_{1}^2 &\leq (C+1) \int_{\frakM[\bda_n-\1]^c} f(\xi)^2 \mu(d\xi) + 2 \int_{\frakM[\bda_n-\1]^c} \D[f, f](\xi) \mu(d\xi) \notag \\
&\leq (C+2) \int_{\frakM[\bda_n-\1]^c} \{ f(\xi)^2+\D[f, f](\xi) \} \mu(d\xi). \notag
\end{align}
Therefore we obtain (\ref{kinjitane}). Thus the proof is completed. 
\end{proof}

%\begin{proof}[Proof of Theorem \ref{Theorem.1}]
We show that the Dirichlet form $(\mathfrak{E}, \mathfrak{D})$ is a quasi-regular Dirichlet form, that is, $(\mathfrak{E}, \mathfrak{D})$ satisfies 
\begin{enumerate}
\item[(C.1)] \label{Q1} There exists an $\mathfrak{E}$-nest consisting of compact sets. 
\item[(C.2)] \label{Q2} There exists an $\| \cdot \|_1$-dense subset of $\mathfrak{D}$ whose elements have $\mathfrak{E}$-continuous $m$-versions.  
\item[(C.3)] \label{Q3} There exist $u_n \in \mathfrak{D}$, $n \in \mathbb{N}$, having $\mathfrak{E}$-continuous $m$-versions $\tilde{u}_n$, and an $\mathfrak{E}$-exceptional set $N$ such that $\{ \tilde{u}_n \}$ separates the points of $X-N$, i.e. for every pair $(s_1, s_2)$ of distinct points of $X-N$, there exists a function $\tilde{u}_n$ which satisfies $\tilde{u}_n(s_1) \neq \tilde{u}_n(s_2)$. 
\end{enumerate}
Please refer to \cite{FOT,MR} for the terminologies in the above.
Let 
\begin{equation*}
\mathfrak{D}_{{\rm cut}} = \{ \chian f; f \in \Dinf, n \in \mathbb{N} \}. 
\end{equation*}
By (\ref{chianfindomain}) we obtain $\mathfrak{D}_{{\rm cut}} \subset \mathfrak{D}$. By (\ref{kinjitane}) and Lemma \ref{teiin} we see $\Dinf \subset \overline{\mathfrak{D}_{{\rm cut}}}$, where $\overline{\mathcal{A}}$ denotes the closure of $\mathcal{A}$ with respect to $\|\cdot\|_1$. By $\overline{\Dinf}=\mathfrak{D}$ we see $\overline{\mathfrak{D}_{{\rm cut}}} = \mathfrak{D}$. Let $\mathfrak{D}(n) = \{ f \in \mathfrak{D}; f = 0 \: \text{a.e.} \: \xi \: \text{on} \: \frakM[2(\bda_{n})_+]^c \}$. By Lemma \ref{chiprop} we see 
\begin{equation*}
\mathfrak{D}_{{\rm cut}} \subset \bigcup_{n=1}^{\infty} \mathfrak{D}(n). 
\end{equation*}
Hence $\{ \frakM[2(\bda_{n})_+] \}_{n \in \N}$ is a compact nest. We thus obtain (C.1) in the above. 
(C.2) and (C.3) are shown by the same argument as that in \cite[p.127]{O96}. 
Combining these we see $(\mathfrak{E}, \mathfrak{D})$ is a quasi-regular Dirichlet form. 
%\end{proof}

%%%%%%%%%%%%%%%%%%%%%%%%%%%%%%%%%%%%%%%%%%%%%%%%%%%%%%%%%%%%%
%%%% ACKNOWLEDGEMENT %%%%%%%%%%%%%%%%%%%%%%%%%%%%%%%%%%%%%%%%
%%%%%%%%%%%%%%%%%%%%%%%%%%%%%%%%%%%%%%%%%%%%%%%%%%%%%%%%%%%%%

\vspace{12pt}

\noindent
{\large{\bf{ACKNOWLEDGEMENT}}} \quad The author would like to express his thanks to Prof. Hideki Tanemura for his valuable suggestion, constant encouragement and many valuable comments. 

%%%%%%%%%%%%%%%%%%%%%%%%%%%%%%%%%%%%%%%%%%%%%%%%%%%%%%%%%%%%%
%%%%%%%%%%%%%%%Reference%%%%%%%%%%%%%%%%%%%%%%%%%%%%%%%%%%%%%
%%%%%%%%%%%%%%%%%%%%%%%%%%%%%%%%%%%%%%%%%%%%%%%%%%%%%%%%%%%%%


\begin{thebibliography}{99}
\bibitem{EsaSDE}
\textsc{S. Esaki and H.Tanemura}, 
Infinite stochastic differential equation of interacting particle systems of long range jumps with long range interactions, in preparation.   

\bibitem{FOT}
\textsc{M. Fukushima, Y. Oshima and M. Takeda}, 
Dirichlet forms and symmetric Markov processes, Walter de Gruyter, 2010. 

\bibitem{KLR}
\textsc{Y. Kondratiev, E. Lytvynov and M. R\"{o}ckner}, 
Equilibrium Kawasaki dynamics of continuous particle systems. 
Infin. Dimen. Anal. Quant. Prob. Rel. Top. {\bf{10}} (2007), no.2, 185--209. 

\bibitem{Kus82}
\textsc{S. Kusuoka}, 
Dirichlet forms and diffusion processes on Banach spaces.
J. Fac. Sci. Univ. Tokyo. Sect. IA Math. 29 (1982), 79--95. 


\bibitem{Lig77}
\textsc{T.M. Liggett}, 
The Stochastic Evolution of Infinite Systems of Interacting Particles. 
\'Ecole d'\'Et\'e de Probabilit\'es de Saint-Flour, VI--1976, pp. 187--248. Lecture Notes in Math. Vol. 598, Springer-Verlag, Berlin, 1977. 

\bibitem{Lig85}
\textsc{T.M. Liggett}, 
Interacting Particle Systems. 
Springer, New York, 1985. 

\bibitem{LO08}
\textsc{E. Lytvynov and N. Ohlerich},
A note on equilibrium Glauber and Kawasaki dynamics for fermion point processes. 
Methods Funct. Anal. Topology 14 (2008), no. 1, 67--80. 

\bibitem{MR} 
\textsc{Z.-M. Ma and M. R\"{o}ckner}, 
Introduction to the theory of (non-symmetric) Dirichlet forms.
Springer-Verlag, Berlin, 1992. 

\bibitem{O96} 
\textsc{H. Osada}, 
Dirichlet form approach to infinite-dimensional Wiener processes with singular interactions.
Comm. Math. Phys. \textbf{176} (1996), 117--131. 

\bibitem{Osa12}
\textsc{H. Osada}, 
Infinite-dimensional stochastic differential equations related to random matrices. Probab. Theory Related Fields {\bf 153} (2012), 471--509. 

\bibitem{Osa13a}
\textsc{H. Osada}, 
Interacting Brownian motions in infinite dimensions with logarithmic interaction potentials. Ann. Probab. {\bf 41} (2013), 1--49. 

\bibitem{Osa13b}
\textsc{H. Osada}, 
Interacting Brownian motions in infinite dimensions with logarithmic interaction potentials II: Airy random point field. Stochastic Process. Appl. {\bf 123} (2013), 813--838. 

\bibitem{OsaShi}
\textsc{H. Osada and T. Shirai}, 
Absolute continuity and singularity of Palm measures of the Ginibre point process.
Probability Theory and Related Fields (Published on line)
DOI 10.1007/s00440-015-0644-6. 

\bibitem{OsaTanCore}
\textsc{H. Osada, and H. Tanemura}, 
Cores of Dirichlet forms related to Random Matrix Theory. 
 Proc. Japan Acad. Ser. A Math. Sci. 90 (2014), 145--150. 

\bibitem{Ru}
\textsc{D. Ruelle}, 
Superstable interactions in classical statistical mechanics.
Commun. Math. Phys. \textbf{18} (1970), 127--159. 

\bibitem{Si} 
\textsc{B. Simon}, 
A canonical decomposition for quadratic forms with applications to monotone convergence theorems.
J. Funct. Anal. \textbf{28} (1978), 377--385. 

\bibitem{Spi69}
\textsc{F. Spitzer}, 
Random Processes Defined Through the Interaction of an Infinite Particle System. 
Springer Lecture Note in Mathematics, Vol. 89, 201--223, Springer Berlin Heidelberg, 1969. 

\bibitem{Tan89}
\textsc{H.Tanemura}, 
Ergodicity for an infinite particle system in $\R^d$ of jump type with hard core interaction,
J. Math. Soc. Japan, \textbf{41}(1989), no.4, 681--697. 

\bibitem{Uem02}
\textsc{T. Uemura}, 
On some path properties of symmetric stable-like processes for one dimension, 
Potential Anal. \textbf{16} (2002), 79--91. 
\end{thebibliography}
\end{document}